\documentclass[10.5pt,reqno]{amsart}
\usepackage{longtable}
\usepackage{hyperref}
\usepackage[T1]{fontenc}
\usepackage[utf8]{inputenc}
\usepackage[english]{babel}
\usepackage{textcomp}
\usepackage{dsfont}
\usepackage{latexsym}
\usepackage{amssymb}
\usepackage{amsthm}
\usepackage{amsmath}
\DeclareMathAlphabet{\mathpzc}{OT1}{pzc}{m}{en}
\usepackage{yfonts}
\usepackage{xfrac}
\usepackage{newlfont}
\usepackage{graphicx}
\usepackage{mathtools}
\usepackage{comment}
\usepackage{indentfirst}
\usepackage{braket}
\usepackage{mathrsfs}
\usepackage{fixmath}

\usepackage{scalerel}[2014/03/10]
\usepackage[usestackEOL]{stackengine}
\newcommand{\dashint}{\,\ThisStyle{\ensurestackMath{%
			\stackinset{c}{.2\LMpt}{c}{.5\LMpt}{\SavedStyle-}{\SavedStyle\phantom{\int}}}%
		\setbox0=\hbox{$\SavedStyle\int\,$}\kern-\wd0}\int}
	
\usepackage{xcolor}

\DeclareMathOperator{\tr}{Tr}

\newcommand{\Aff}{\mathrm{Aff}}
\newcommand{\Hol}{\mathrm{Hol}}
\newcommand{\Min}{\mathrm{min}}
\newcommand{\Max}{\mathrm{max}}

\newcommand{\ee}{\mathrm{e}}

\newcommand{\vect}[1]{\mathbf{{#1}}}
\newcommand{\dd}{\mathrm{d}}

\DeclarePairedDelimiter{\abs}{\lvert}{\rvert}

\DeclarePairedDelimiter{\norm}{\lVert}{\rVert}

\let\originalleft\left
\let\originalright\right
\renewcommand{\left}{\mathopen{}\mathclose\bgroup\originalleft}
\renewcommand{\right}{\aftergroup\egroup\originalright}

\newcommand{\N}{\mathds{N}}
\newcommand{\Z}{\mathds{Z}}

\newcommand{\C}{\mathds{C}}
\newcommand{\Hd}{\mathds{H}}
\newcommand{\Od}{\mathds{O}}
\newcommand{\R}{\mathds{R}}

\newcommand{\T}{\mathds{T}}

\newcommand{\Ff}{\mathfrak{F}}

\newcommand{\Hs}{\mathscr{H}}
\newcommand{\Us}{\mathscr{U}}

\newcommand{\Ac}{\mathcal{A}}
\newcommand{\Bc}{\mathcal{B}}
\newcommand{\Cc}{\mathcal{C}}
\newcommand{\Dc}{\mathcal{D}}

\newcommand{\Fc}{\mathcal{F}}

\newcommand{\Hc}{\mathcal{H}}
\newcommand{\Ic}{\mathcal{I}}

\newcommand{\Kc}{\mathcal{K}}
\newcommand{\Lc}{\mathcal{L}}
\newcommand{\cM}{\mathcal{M}}
\newcommand{\Nc}{\mathcal{N}}

\newcommand{\Pc}{\mathcal{P}}

\newcommand{\Rc}{\mathcal{R}}

\newcommand{\Uc}{\mathcal{U}}

\newcommand{\Wc}{\mathcal{W}}

\renewcommand{\Im}{\mathrm{Im}\,}
\renewcommand{\Re}{\mathrm{Re}\,}

\newcommand{\meg}{\leqslant}
\newcommand{\Meg}{\geqslant}

\renewcommand{\phi}{\varphi}
\newcommand{\mi}{\mu}

\newcommand{\leftexp}[2]{{\vphantom{#2}}^{#1}{#2}} 
\newcommand{\trasp}{\leftexp{t}}
\newcommand{\Lin}{\mathscr{L}}

\title[Invariant Spaces of Holomorphic Functions]{A Survey on Invariant Spaces of Holomorphic Functions on Symmetric Domains}

\date{}

\begin{document}

\theoremstyle{definition}
\newtheorem{definition}{Definition}[section]

\newtheorem{remark}[definition]{Remark}

\theoremstyle{plain}
\newtheorem{theorem}[definition]{Theorem}

\newtheorem{lemma}[definition]{Lemma}

\newtheorem{proposition}[definition]{Proposition}

\newtheorem{corollary}[definition]{Corollary}

\author[M. Calzi]{Mattia Calzi}

\address{Dipartimento di Matematica, Universit\`a degli Studi di
	Milano, Via C. Saldini 50, 20133 Milano, Italy}
\email{{\tt mattia.calzi@unimi.it}}

\thanks{{\em Math Subject Classification 2020} 46E15, 47B33, 32M15  }
\thanks{The author is a member of the 	Gruppo Nazionale per l'Analisi Matematica, la Probabilit\`a e le	loro Applicazioni (GNAMPA) of the Istituto Nazionale di Alta	Matematica (INdAM). The author was partially funded by the INdAM-GNAMPA Project CUP\_E55F22000270001.} 

\begin{abstract}
	We present some old and new results on a class of invariant spaces of holomorphic functions on symmetric domains, both in their circular bounded realizations and in their unbounded realizations as  Siegel domains of type II. These spaces include: weighted Bergman spaces; the Hardy space $H^2$; the Dirichlet space; holomorphic Besov spaces; the Bloch space.

	Our main focus will be on invariant Hilbert and semi-Hilbert spaces, but we shall also discuss minimal and maximal spaces in suitable classes of invariant Banach and semi-Banach spaces.
\end{abstract}
\maketitle

\section{Introduction}

Let $D$ denote the unit disc in $\C$, and consider the group $G(D)$ of its biholomorphisms, that is,
\[
G(D)=\Set{z\mapsto \alpha\frac{z-b}{1-\overline b z}\colon \alpha\in \T, \abs{b}<1}.
\]
Then, $G(D)$ is a connected simple Lie group; denote by $\widetilde G(D)$ its universal covering.
Consider the representations $\widetilde U_\lambda$ of $\widetilde G(D)$ in the space $\Hol(D)$ of holomorphic functions on $D$ defined by
\[
\widetilde U_\lambda(\phi) f= (f\circ \phi^{-1}) (J\phi^{-1})^{\lambda/2}
\]
for every $\lambda\in \C$, for every $\phi\in \widetilde G$, and for every $f\in \Hol(D)$, where $\phi^{-1}$ acts on $D$ by means of the canonical mapping $\widetilde G(D)\to G(D)$, while $(J\phi^{-1}(z))^{\lambda/2}$ is defined as a continuous function of $(\lambda, \phi, z)\in \C\times \widetilde G(D)\times D$ (hence analytic in $(\lambda,\phi,z)$ and holomorphic in $(\lambda,z)$). Then, it is known (cf.~\cite{VergneRossi}) that the holomorphic discrete series representations of $\widetilde G(D)$, that is, the unitary representations of $\widetilde G(D)$ which are unitarily equivalent to a subrepresentation of the left regular representation in $L^2(\widetilde G(D))$, may be identified with the irreducible unitary representations induced by the $\widetilde U_\lambda$, for $\lambda>1$, in the weighted Bergman spaces
\[
A^2_\lambda(D)=\Set{f\in \Hol(D)\colon \int_{D} \abs{f(z)}^2 (1-\abs{z}^2)^{\lambda-2}\,\dd z}.
\]
The unweighted Bergman space corresponds to $\lambda=2$. Observe that $A^2_\lambda(D)$ embeds continuously into $\Hol(D)$, so that it is a repreducing kernel Hilbert space. More precisely, if we define the function
\[
\Kc^{(\lambda)}\colon D\times D\ni (z,w)\mapsto c_\lambda (1-z \overline w)^{-\lambda}\in \C
\]
for a suitable $c_\lambda>0$, then 
\[
\langle f \vert \Kc^{(\lambda)}_z\rangle_{A^2_\lambda(D)}=f(z)
\]
for every $f\in A^2_\lambda(D)$ and for every $z\in D$, where $\Kc^{(\lambda)}_z=\Kc^{(\lambda)}(\,\cdot\,,z)$. The sesquiholomorphic function $\Kc^{(\lambda)}$ is then called the reproducing kernel of $A^2_\lambda(D)$. 
Notice that saying that $\widetilde U_\lambda(\phi)$ induces a unitary automorphism of $A^2_\lambda(D)$ is equivalent to saying that
\begin{equation}\label{eq:5}
	\Kc^{(\lambda)}=(\widetilde U_\lambda (\phi)\otimes \overline{\widetilde U_\lambda(\phi)})\Kc^{(\lambda)}.
\end{equation}

Notice, in addition, that the Hardy space 
\[
H^2(D)=\Set{f\in \Hol(D)\colon \sup_{r\in (0,1)} \int_{\T} \abs{f(r \alpha)}^2\,\dd \alpha<\infty}
\]
is also a reproducing kernel Hilbert space, with reproducing kernel
\[
\Kc^{(1)}\colon D\times D\ni (z,w)\mapsto c_1 (1-z \overline w)^{-1}\in \C
\]
for a suitable $c_1>0$. Now, suitably raising both sides of~\eqref{eq:5} to the power $1/\lambda$  yields  $\Kc^{(1)}=(\widetilde U_1 (\phi)\otimes \overline{\widetilde U_1(\phi)})\Kc^{(1)}$ for every $\phi \in \widetilde G(D)$, so that $H^2(D)$ is $\widetilde U_1$-invariant with its norm. 

One may then wonder for which values of $\lambda\in \C$ the sesquiholomorphic function
\[
\Kc^\lambda\colon   D\times D\ni (z,w)\mapsto   (1-z \overline w)^{-\lambda}\in \C
\]
is the reproducing kernel of some reproducing kernel Hilbert space of holomorphic functions. It turns out that this is equivalent to a certain positivity condition on $\Kc^\lambda$, which is satisfied if and only if $\lambda\Meg 0$. Given $\lambda\Meg 0$, the corresponding reproducing kernel Hilbert space $H_\lambda(D)$ may then be defined as the  completion of the space of finite linear combinations of the $\Kc^\lambda_z$, endowed with the unique scalar product such that
\[
\langle \Kc^\lambda_z \vert \Kc^\lambda_w\rangle= \Kc^\lambda(z,w)=(1-z\overline w)^{-\lambda}
\]
for every $z,w\in D$. It turns out that $H_\lambda(D)$ actually embeds continuously into $\Hol(D)$, so that it is a reproducing kernel Hilbert space of holomorphic functions; in addition, $\widetilde U_\lambda$ induces an irreducible unitary representation of $\widetilde G(D)$ in $H_\lambda(D)$ (cf.~\cite{VergneRossi}). 

There is also an interesting converse result to the previous construction. Namely, if $H$ is a reproducing kernel Hilbert space  of holomorphic functions in which $\widetilde U_\lambda$ induces a (automatically continuous) bounded representation, then $\lambda \Meg 0$ and $H=H_\lambda(D)$ with equivalent norms. If, in addition, $U_\lambda$ induces a unitary representation in $H$, then $H$ and $H_\lambda(D)$ have proportional norms (cf.~\cite[Theorem 5.1]{Arazy}).

A general abstract study of the Banach spaces continuously embedded in $\Hol(D)$ and in which $\widetilde U_\lambda$ induces a bounded representation for some $\lambda>0$ was more recently developed in~\cite{AlemanMas}. In this paper, besides density and duality (and other) problems,   minimal and maximal  spaces (subject to certain further conditions) were determined. Namely, two $\widetilde U_\lambda$-invariant Banach spaces $X_\Min^\lambda$ and $X_\Max^\lambda$ were determined such that $X_\Min^\lambda \subseteq X\subseteq X_\Max^\lambda$ continuously for every $\widetilde U_\lambda$-invariant Banach space $X$ satisfying certain conditions. Explicitly, $X_\Min^\lambda$  may be realized as the weighted Bergman space of type $L^1$
\[
A^1_\lambda(D)=\Set{f\in \Hol(D) \colon \int_{D} \abs{f(z)}(1-\abs{z}^2)^{\lambda/2-2}\,\dd z<\infty}
\]
(when $\lambda>2$), whereas  $X_\Max^\lambda$  may be realized as the weighted Bergman  space of type $L^\infty$
\[
A^\infty_\lambda(D)=\Set{f\in \Hol(D)\colon \sup_{z\in D} (1-\abs{z}^2)^{\lambda/2}\abs{f(z)}<\infty}.
\]
This analysis extends the analogous analysis performed before in~\cite{RubelTimoney,ArazyFisherPeetre}, where the $\widetilde U_0$-invariant maximal space  $X_\Max^0$ was identified with the Bloch space $\Bc(D)$, defined as
\[
\Set{f\in \Hol(D)\colon \sup_{z\in D}(1-\abs{z}^2)\abs{f'(z)}<\infty}=\Set{f\in \Hol(D)\colon f'\in A^\infty_2(D)},
\]
whereas the minimal space $X_\Min^0$ was identified with the holomorphic Besov space
\[
B^1(D)=\Set{f\in \Hol(D)\colon  \int_{D} \abs{f''(z)}\,\dd z<\infty}=\Set{f\in \Hol(D)\colon f''\in A^1_4(D)}.
\]
It is worthwhile underlying that, whereas the Bloch space is endowed with the corresponding  seminorm, and is therefore \emph{not} Hausdorff (otherwise, $\widetilde U_0$ would not induce a \emph{bounded} representation in $\Bc(D)$), the space $B^1(D)$ is actually endowed with a norm (e.g., $f\mapsto \abs{f(0)}+\abs{f'(0)}+\norm{f''}_{A^1_4(D)}$), and one may find an equivalent norm which is also $\widetilde U_0$-invariant. 
We would like to emphasize that the problem of dealing with non-Hausdorff spaces raises the issue of finding some natural conditions on a given space $X$ which replace the no longer available continuity of the canonical inclusion $X\subseteq \Hol(D)$. It turns out that, as long as one is concerned with maximal spaces, a reasonable condition to impose is the existence of a \emph{decent} continuous linear functional on $X$, that is, a non-zero continuous linear functional on $X$ which is continuous with respect to the topology of compact convergence; equivalently, a non-zero continuous linear functional on $X$ is decent if it extends to a continuous linear functional on $\Hol(D)$. We shall also say that $X$ is decent if it admits a decent continuous linear functional. Then, $\Bc(D)$ may be characterized as the maximal \emph{decent} semi-Banach space (i.e., a complete space whose topology is defined by a seminorm) of holomorphic functions in which $\widetilde U_0$ induces a bounded representation. 

Since the minimal space is always a Banach space (for $\lambda\Meg 0$), one may avoid these issues and consider only Banach spaces continuously embedded in $\Hol(D)$ when dealing with the minimal space for $\lambda\Meg 0$. Notice, though, that $B^1(D)$ is minimal only in the class of the spaces which contain non-constant functions; the space of constant functions is the actual (but rather uninteresting) `minimal' space. If one still wishes to consider semi-Banach spaces, then the decency assumption does not seem to be sufficient, as one might, for instance, find $\widetilde U_0$-invariant algebraic complements of the space of constant functions in $\Bc(D)$, which would then be invariant Banach spaces \emph{which do not embed continuously into $\Hol(D)$} but are still decent.\footnote{We do not know whether such spaces exist or not.} One then has to consider only decent spaces which contain the space of constant functions, and then $B^1(D)$ is minimal among those ones in which $\widetilde U_0$ induces a continuous bounded representation.

Let us now mention explicitly that, if $\lambda\in-\N$, then  the $(1-\lambda)$-th derivative intertwines $\widetilde U_\lambda$ and $\widetilde U_{2-\lambda}$, that is,
\[
(\widetilde U_\lambda(\phi) f)^{(1-\lambda)}= \widetilde U_{2-\lambda}(\phi)f^{(1-\lambda)}
\]
for every $\phi\in \widetilde G(D)$ and for every $f\in \Hol(D)$. Using this fact, one may then extend the preceding analysis and show that the Dirichlet space (for $\lambda=0$) and, more generally, the spaces
\[
\widetilde H_\lambda=\Set{f\in \Hol(D)\colon f^{(1-\lambda)}\in H_{2-\lambda}},
\]
for $\lambda\in -\N$, are decent semi-Hilbert spaces in which $\widetilde U_\lambda$ induces a continuous unitary representation, and that suitable maximal and minimal spaces are obtained considering
\[
\Set{f\in \Hol(D)\colon \sup_{z\in D} (1-\abs{z}^2)^{1-\lambda/2} \abs{f^{(1-\lambda)}(z)}<\infty }
\]
and
\[
\Set{f\in \Hol(D)\colon f^{(1-\lambda)}\in X_\Min^{2-\lambda}}.
\]

The preceding problems were also considered in more general contexts. First, replacing $D$ with the unit ball in $\C^n$ (cf., e.g.,~\cite{ArazyFisherPeetre,Peloso}). Then, more generally, replacing $D$  with a general (irreducible) symmetric domain, of which the unit disc in $\C$ is the simplest example. Notice that the unit balls in the various $\C^n$ are precisely the only (irreducible) symmetric domains which are \emph{strongly} pseudoconvex: this fact marks an inportant difference between these `rank $1$' domains and the other ones, which typically exhibit a more complicated (and interesting) behaviour. See Section~\ref{sec:1} for more information on symmetric domains. Here we simply mention that the group of biholomorphisms $G(D)$ of a bounded symmetric domain acts transitively on $D$, that is, $D$ is a homogeneous domain (cf.~\cite[No.~17]{Cartan}). 
Cf.~\cite{Timoney1,Timoney2,ArazyFisher4,ArazyFisherPeetre,ArazyUpmeier3,AlemanMas} for the study of invariant (semi-)Banach spaces of holomorphic functions (in particular, minimal and maximal spaces), and~\cite{Peetre,Peetre2,ArazyFisher,Peetre3,Fisher,ArazyFisherPeetre2,ArazyFisher3,ArazyFisher2,Zhu,Arazy2,Arazy3,Arazy,ArazyUpmeier,ArazyUpmeier4,ArazyUpmeier2} for the study of invariant (semi-)Hilbert spaces of holomorphic functions on bounded symmetric domains. See also~\cite{BB,Ishi6} for the study of invariant Hilbert spaces of holomorphic functions on bounded homogeneous domains.

These problems were also studied in the unbounded realizations of bounded symmetric domains as Siegel domains, of which the upper half-plane and the Siegel upper half-spaces
\[
\Set{(\zeta,z)\in \C^n\times \C\colon \Im z-\abs{\zeta}^2>0}
\]
are the simplest examples, and correspond to the unit disc in $\C$ and to the unit ball in $\C^{n+1}$, respectively. When $D$ is replaced by its unbounded realization $\Dc$ as a Siegel domain, the group $\Aff(\Dc)$ of affine automorphisms of  $\Dc$ acts transitively on $\Dc$ and is somewhat easier to work with, so that several results in this constext are obtained actually studying the $\Aff(\Dc)$-invariant spaces of holomorphic functions. Cf.~\cite{Ishi3,Ishi4,Ishi5,Garrigos,Arcozzietal,Rango1,Tubi} for the study of invariant (semi-)Banach spaces of holomorphic functions on symmetric Siegel domains.

In this paper we shall survey some old and recent advances in the study of the above and related problems on symmetric domains, both in their bounded and unbounded realizations. 

In Section~\ref{sec:1}, we shall review some basic facts about bounded symmetric domains and symmetric cones. We shall follow the formalism of Jordan triple systems, as it seems the most suitable one for a concise exposition of the subject. In particular, we shall provide a brief description of irreducible symmetric domains, both in their circular and bounded realizations, and in their unbounded realizations as Siegel domains of type II.

In Section~\ref{sec:2}, we shall classify the $\widetilde U_\lambda$-invariant \emph{closed} vector subspaces of $\Hol(D)$. This preliminary study is actually of fundamental importance in what follows.
In Section~\ref{sec:3}, we shall discuss the problem of dealing with invariant semi-Banach and semi-Hilbert spaces, and discuss the various approaches that have been pursued.

In Section~\ref{sec:4} we shall then describe a classification of $\widetilde U_\lambda$-invariant semi-Hilbert spaces of holomorphic functions. In Section~\ref{sec:5} we shall indicate some of the various alternative descriptions of the aforementioned spaces that have been developed in the literature.

In Section~\ref{sec:6}, we shall transfer the preceding results to the unbounded realization of $D$ as a Siegel domain $\Dc$ of type II, and also describe the classification of $\Aff(\Dc)$-invariant semi-Hilbert spaces whenever possible.
Finally, in Section~\ref{sec:7} we shall consider the problem of finding minimal and maximal spaces in suitable families of $\widetilde U_\lambda$-invariant semi-Banach spaces. We shall mainly pursue this analysis on $\Dc$ instead of $D$ for some technical advantages.

\section{Bounded Symmetric Domains and Jordan Triple Systems}\label{sec:1}

A bounded symmetric domain is a bounded connected open subset $D$ of $\C^n$ such that for every $z\in D$ there is an involutive biholomorphism of $D$ having $z$ as an isolated fixed point (or, equivalently, as its unique fixed point). It is then known that $D$ is a homogeneous domain, that is, that $D$ admits a transitive group of biholomorphism (cf.~\cite[17]{Cartan}). In addition, the component of the identity $G_0(D)$ of the group of biholomorphisms  $G(D)$ of $D$ is a semisimple Lie group, and is simple if and only if $D$ is irreducible, that is, is not biholomorphic to the product of two (non-trivial) symmetric domains.

Every bounded symmetric domain is biholomorphic to a circular (hence convex, cf.~\cite{Cartan} or also~\cite[Theorem 1.6]{Loos}) bounded symmetric domain, which is unique up to linear isomorphisms. This realization of $D$ may be effectively studied either by means of the study of the group $G_0(D)$ and its Lie algebra (cf., e.g.,~\cite{Helgason,FarautKoranyi}) or by means of the theory of Jordan triple systems (or by means of the theory of Jordan pairs, cf., e.g.,~\cite{Loos,Satake}). In this section, we shall briefly describe the second approach, which seems to better highlight the algebraic objects we shall need.

\begin{definition}
	A   hermitian Jordan triple system  $Z$ is a finite dimensional complex vector space endowed with a triple product $\{\,\cdot\,,\,\cdot\,,\,\cdot\,\}\colon Z\times Z\times Z\to Z$ which is $\R$-trilinear, $\C$-bilinear and symmetric in the first and third argument,   $\C$-antilinear in the second argument, and satisfies the identity
	\[
	\{u,v,\{x,y,z\}\}-\{x,y,\{u,v,z\}\}=\{\{u,v,x\},y,z\}-\{x,\{v,u,y\},z\}
	\]
	for every $u,v,x,y,z\in Z$.
	
	We say that $Z$ is positive if every $\lambda\in \C$ such that   $\{x,x,x\}=\lambda x$ for some non-zero $ x\in Z$ is necessarily $>0$.
	In this case, we endow $Z$ with the scalar product defined by
	\[
	\langle x\vert y \rangle\coloneqq \tr \{x,y,\,\cdot\,\}
	\]
	for every $x,y\in Z$ (cf.~\cite[Corollary 3.16]{Loos}).
\end{definition}

Notice that, if we define $D(u,v)\coloneqq \{u,v,\,\cdot\,\}$, then  $D(u,v)^*=D(v,u)$ with respect to the chosen scalar product, thanks to~\cite[Proposition 3.4 and Corollary 3.16]{Loos}), and the condition defining a hermitian Jordan triple system becomes
\[
D(u,v)\{x,y,z\}=\{D(u,v)x,y,z\}-\{x,D(u,v)^*y,z\}+\{x,y,D(u,v)z\}
\]
for every $u,v,x,y,z\in Z$, which expresses the fact that $D(u,v)$ is a  derivation of $Z$. 

\begin{definition}
	We say that $x\in Z$ is  is a tripotent if $x=\{x,x,x\}$. 
	
	Two tripotents $x,y\in Z$ are said to be orthogonal if one of the following equivalent conditions hold: $\{x,x,y\} =0$; $\{y,y,x\}=0$; $D(x,y)=0$; $D(y,x)=0$ (cf.~\cite[Lemma 3.9]{Loos}). If $x,y\in Z$ are orthogonal tripotents, then also $x+y$ is a tripotent.
	
	A tripotent $x\in Z$ is   primitive if it cannot be written as the sum of two non-trivial orthogonal tripotents. 
	A tripotent $x\in Z$ is maximal if there are no non-trivial tripotents which are orthogonal to  $x$.
\end{definition}

We are now able to describe the correspondence between bounded symmetric domains and Jordan triple systems (cf.~\cite[Theorems 2.10, 3.17, and 4.1]{Loos}). Notice that this correspondence is actually bijective (cf.~\cite[Theorem 4.1]{Loos}).

\begin{proposition}\label{prop:2}
	Let $Z$ be a positive hermitian Jordan triple system, and set $\abs{x}_Z\coloneqq \sqrt{\norm{D(x,x)}}$ for every $x\in Z$. Then, $\Set{x\in Z\colon \abs{x}_Z<1}$ is a circular bounded symmetric domain.
	
	Let $D$ be a circular bounded symmetric domain in a finite-dimensional complex vector space $Z$.  Denote by $\Kc$ the unweighted Bergman kernel on $D$ and endow $Z$ with the scalar product given by the Bergman metric at $0$.\footnote{In other words,  $\Kc$ is the reproducing kernel of the Bergman space $\Hol(D)\cap L^2(D)$, so that $\Kc\colon D\times D\to \C$ is a sesquiholomorphic map and $f(z)=\langle f\vert \Kc(\,\cdot\,,z)\rangle_{L^2(D)}$ for every $f\in \Hol(D)\cap L^2(D)$ and for every $z\in D$. Then, we endow $Z$ with the scalar product  $(v,w)\mapsto \partial_{1,v}\overline{\partial_{2,w}} \log \Kc(0,0)$, where, $\partial_{1,v}$ denotes the holomorphic dervative in the direction of $v$ in the first argument, whereas $\overline{\partial_{2,w}}$ denotes the antiholomorphic derivative in the direction of $w$ in the second argument.} Then, the equality
	\[
	\langle\{ x,y,z \}\vert w\rangle\coloneqq\frac 12  \partial_{1,x}\overline{\partial_{2,y}} \partial_{1,z} \overline{\partial_{2,w}} \log \Kc(0,0),
	\]
	for every $x,y,z,w\in Z$, defines the structure of a positive hermitian Jordan triple system on $Z$, for which $D=\Set{x\in Z\colon \abs{x}_Z<1}$, the stabilizer $K$ of $0$ in $G(D)$ is the group of automorphisms of $Z$, and
	\[
	\langle x\vert y \rangle=\tr D(x,y)
	\]
	for every $x,y\in Z$.
\end{proposition}

\begin{proof}
	The first assertion follows from~\cite[Theorems 3.17 and 4.1]{Loos}. The second assertion follows from~\cite[Theorem 2.10 and Corollary 3.16]{Loos} and from the fact that  $K$ is the group of linear biholomorphisms of $D$ (cf., e.g.,~\cite[1.5]{Loos}), so that every element of $K$ is an automorphism of $Z$ by Proposition~\ref{prop:2}, whereas every automorphism of $Z$ preserves $D$ by~\cite[Theorem 3.17]{Loos}. 
\end{proof}

We shall now pass to describing the unbounded realizations of a bounded symmetric domain as Siegel domains of type II. It is convenient  to first  introduce the notion of a Jordan algebra.

\begin{definition}
	A (real or complex) Jordan algebra $A$ is a commutative, not necessarily associative  (real or complex) algebra such that $x^2(x y)=x(x^2y)$ for every $x,y\in A$.  A real Jordan algebra $A$ is formally real if $x=y=0$ for every $x,y\in A$ such that $x^2+y^2=0$.
	
	Given a real Jordan algebra $A$, its complexification $A_\C$ (as a vector space), endowed with the triple product defined by
	\[
	\{x,y,z\}\coloneqq x (\overline y z)-(x z) \overline y+z( \overline y x),
	\]
	is a hermitian Jordan triple system which is called the hermitification of $A$. 
	
	Two idempotents $x,y$ of a Jordan algebra $A$ are said to be orthogonal if $xy=0$. In this case, $x+y$ is an idempotent as well.
\end{definition}

Notice that $A$ is a formally real Jordan algebra if and only if $A_\C$ is a positive hermitian Jordan triple system (cf.~\cite[3.7 and 3.13]{Loos}). In addition, if $e,e'$ are orthogonal idempotents of $A$, then $e,e'$ are orthogonal tripotents of $A_\C$ (and conversely).

See~\cite[Theorem 3.13]{Loos} for a proof of the following result (`Peirce decomposition').

\begin{proposition}
	Let $Z$ be a positive hermitian Jordan triple system, and let $e$ be a tripotent in $Z$. Define $Z_\alpha(e)\coloneqq \Set{x\in Z\colon D(e,e)x=\alpha x}$ for $\alpha\in \R$. Then, the $Z_\alpha(e)$ are pairwise orthogonal,
	\[
	Z=Z_0(e)\oplus Z_{1/2}(e)\oplus Z_1(e) \qquad \text{and} \qquad
	\{Z_\alpha(e),Z_{\beta}(e),Z_\gamma(e)\}\subseteq Z_{\alpha-\beta+\gamma}(e)
	\]
	for every $\alpha,\beta,\gamma\in \R$. The space $Z_0(e)$ is trivial if and only if $e$ is a maximal tripotent.
	
	In addition, $Z_1(e)$ is a complex Jordan algebra with product $(x,y)\mapsto \{x,e,y\}$ and unit $e$, and coincides (as a Jordan triple system) with the hermitification of the formally real Jordan algebra $A(e)\coloneqq \Set{z\in Z_2(e)\colon z=z^*}$, where 
	\[
	z^*\coloneqq \{e,z,e\}
	\]
	for every $z\in Z_1(e)$.
\end{proposition}

We now fix a maximal tripotent $e$ in $Z$ and a frame $(e_1,\dots, e_r)$ in the formally real Jordan algebra $A(e)\subseteq Z_1(e)$, that is, a  family of pairwise orthogonal primitive idempotents such that $e=e_1+\cdots+e_r$. Note that $r$ is independent of the chosen frame and is called the rank of $A$ (and $Z$), cf.~\cite[Theorem III.1.2]{FarautKoranyi}.
Then, we may consider the complex Jordan algebra  $Z_1(e_1+\cdots+e_j)$, for $j=1,\dots, r$, and denote by $\Delta_j$ its (polynomial) determinant function, so that
\[
\Delta_j(z)={\det}_\C( \C[z]\ni y\mapsto y z\in \C[z]),
\]
where $\C[z]$ denotes the (commutative and) associative subalgebra $\C[z]$ of $Z_1(e_1+\cdots+e_j)$ generated by $e_1+\cdots+e_j$ and $z$ (cf.~\cite[p.~29]{FarautKoranyi}). If $z\in Z_1(e)$ and $\Delta_r(z)\neq 0$, then $z$ is said to be invertible in $Z_1(e)$, and its inverse $z^{-1}$ is the (unique) inverse of $z$ in the subalgebra $\C[z]$ of $Z_1(e)$  generated by $z$ and $e$.  We shall provide later on a description of these functions in series of examples.  By means of the orthogonal projection of $Z$ onto $Z_1(e_1+\cdots+e_j)$, we may then extend $\Delta_j$ to a polynomial function on $Z$. Then, we define
\[
\Delta^{\vect s}\coloneqq \Delta_1^{s_1-s_2}\cdots \Delta_{r-1}^{s_{r-1}-s_r}\Delta_r^{s_r}
\]
for every $\vect s\in \C^r$. These functions are rational (exactly) when $\vect s\in \Z^r$ and are polynomial when (and only when, if $D$ or, equivalently, $\Omega$, is irreducible) $s_1\Meg \cdots \Meg s_r$ (cf.~\cite[Proposition 2.1]{Ishi2} and~\cite[Proposition XI.2.1]{FarautKoranyi}). We denote by $\N_\Omega$ the set of $\vect s\in \R^r$ such that $\Delta^{\vect s}$ is polynomial.

Notice that the $\Delta^{\vect s}$ are  in  priciple only defined on the convex domain $Z_{1/2}(e)+(\Omega+ i  A(e))$, where 
\[
\Omega\coloneqq \Set{x^2\colon x\in A(e), \det x\neq 0}.
\]
Notice that $\Omega$ is a symmetric cone, that is, it is open, convex, self-adjoint with respect to the chosen scalar product (i.e., $\Omega=\Set{x\in A(e)\colon \forall y\in \overline\Omega\setminus\Set{0}\:\: \langle x,y\rangle>0}$) and homogeneous (i.e., admits a transitive group of linear isomorphisms), cf.~\cite[Theorem III.2.1]{FarautKoranyi}.
We shall also write $\Delta$ instead of $\Delta_r$. 

We may then consider the unbounded realization $\Dc$ of $D$ as a Siegel domain associated with the maximal tripotent $e$ (cf.~\cite[Section 10]{Loos}). Namely, define 
\[
\Phi\colon Z_{1/2}(e)\times Z_{1/2}(e)\ni (\zeta,\zeta')\mapsto 2\{\zeta,\zeta',e\}\in Z_1(e),
\]
so that $\Phi$ is non-degenerate, $^*$-hermitian, and $\overline{\Omega}$-positive,\footnote{In other words, $\Phi(\zeta,\zeta')^*=\{e,\Phi(\zeta,\zeta'),e\}=\Phi(\zeta',\zeta)$ for every $\zeta,\zeta'\in Z_{1/2}(e)$, and $\Phi(\zeta,\zeta)\in \overline \Omega\setminus \Set{0}$ for every non-zero $\zeta\in Z_{1/2}(e)$. } and set (with $\Phi(\zeta)=\Phi(\zeta,\zeta)$ for simplicity)
\[
\Dc\coloneqq \Set{(\zeta,z)\in Z_{1/2}(e)\times Z_1(e)\colon \frac{z-z^*}{2i}-\Phi(\zeta)\in \Omega}.
\]
When $Z_{1/2}(e)=\Set{0}$, so that $\Dc=A(e)+i  \Omega$, $\Dc$ is called a tube domain (or a Siegel domain of type I), and $D$ is said to be of tube type.
Notice that $D$ is canonically biholomorphic to $\Dc$ by means of the Cayley transform
\[
\Cc\colon D\ni z\mapsto i  (e-z_1)^{-1}(e+z_1) +2 (e-z_1)^{-1}z_{1/2}\in \Dc;
\]
here, we endow $Z$ with the  (complex) Jordan algebra structure with product $(z,z')\mapsto \{z,e,z'\}$.

Let us observe explicitly that,  if we define the genus $g$ of $D$ as $(\dim_\C Z_{1/2}(e)+2\dim_\C Z_1(e))/r $ (which does \emph{not} depend on the choice of $e$), then the unweighted Bergman kernel on $\Dc$, that is, the reproducing kernel of $\Hol(\Dc)\cap L^2(\Dc)$, is 
\[
((\zeta,z),(\zeta',z'))\mapsto c' \Delta^{-g}\left( \frac{z-\overline{z'}}{2 i }- \Phi(\zeta,\zeta') \right)
\]
for some $c'>0$. We also define $a\in \N$ so that $\frac 1 r \dim Z_1(e)-1= a \frac{r-1}{2}$. Notice that $a$ does \emph{not} depend on the choice of $e$.

We now describe the list of simple positive hermitian Jordan triple systems (up to isomorphisms), that is, those positive hermitian Jordan triple systems which are not isomorphic to the direct sum of two non-trivial hermitian Jordan triple systems (cf.~\cite[4.14 and 4.17]{Loos}). This description will provide a description of the corresponding irreducible bounded symmtric domains by Proposition~\ref{prop:2}. The original classification of irreducible bounded symmetric domains may be found in~\cite{Cartan}. 

\subsection{Domains of Type $(I_{p,q})$}

In this case, $Z$ is the space $M_{p,q}(\C)$ of complex $p\times q$ matrices ($p\meg q$), with triple product given by  $\{x,y,z\}=\frac 1 2 (x y^* z+z y^* x)$. The associated bounded symmetric domain is $\Set{z\in Z\colon I_p-zz^* \text{ is positive and non-degenerate}}=\Set{z\in Z\colon \norm{z}<1}$, where $I_p$ is the $p\times p$ identity matrix and $\norm{z}$ denotes the   norm of $z$ considered as a linear mapping $\C^p\to \C^q$. 
One may then choose $e\coloneqq (I_p, 0)$ as a maximal tripotent. Then, $Z_1(e)=\Set{(z,0)\colon z\in M_{p,p}(\C)}$, $Z_{1/2}(e)=\Set{(0,\zeta)\colon \zeta\in M_{p,q-p}(\C)}$ and $A(e)=\Set{(z,0)\in Z_1(e)\colon z=z^*  }$.

The cone $\Omega$ is then the cone of non-degenerate positive hermitian complex $p\times p$ matrices, while
\[
\Phi((0,\zeta),(0, \zeta'))= (\zeta\zeta'^*,0)
\]
for every $\zeta,\zeta'\in M_{p,q-p}(\C)$, and
\[
\Cc(z,\zeta)=(i (I_p-z)^{-1}(I_p+z), (I_p-z)^{-1} \zeta  )
\]
for every $z\in M_{p,p}(\C)$ and for every $\zeta\in M_{p, q-p}(\C)$. 
In this case, choosing the frame $(e_1,\dots, e_p)$ of $A(e)$ given by $e_j=(\delta_{j,u}\delta_{j,v})_{u,v}$ for every $j=1,\dots, p$, the polynomial function
\[
\Delta_j(z,\zeta)
\]
is simply the (complex) minor corresponding to the first $j$ rows and columns of $z$. In particular, $\Delta(z,\zeta)$ is the (complex) determinant of $z$.

In particular, the rank of $D$ is $p$, $a=0$ if $p=1$, whereas $a=2$ for $p\Meg 2$, the genus of $D$ is $q+p$, and $D$ is of tube type if and only if $p=q$.

Notable examples arise when $q=p+1$, in which case $D$ becomes the unit ball in $\C^q$ and  $\Dc$ becomes the Siegel upper half-space $\Set{(\zeta,z)\in \C^p\times \C\colon \Im z-\abs{\zeta}^2>0 }$.

\subsection{Domains of Type $(I\!I_{n})$, $n\Meg 2$}

In this case, $Z$ is the space $A_{n}(\C)$ of skew-symmetric complex $n\times n$ matrices, with triple product given by  $\{x,y,z\}=\frac 1 2 (x y^* z+z y^* x)$. The associated bounded symmtric domain is then $\Set{z\in Z\colon \norm{z}<1}$. 

Denote by $H_p(\Hd)$ the formally real Jordan algebra of hermitian $p\times p$ matrices over the division ring of the quaternions, with the symmetrized product $(x,y)\mapsto \frac 1 2 (x y+y x)$. Observe that $\Hd$ may be identified (as a division ring) with a subset of $M_{2,2}(\C)$ by means of the mapping
\[
a+b\vect i +c\vect j+d\vect k\mapsto \left(\begin{matrix} a+i  d & b+i  c\\ -b+i  c & a-i  d  \end{matrix}\right) ,
\] 
so that $\vect i H_{p}(\Hd)$ may be identified with a subspace of $A_{2 p}(\C) $, and $A_{2p}(\C)$ is the hermitification of $\vect i H_p(\Hd)$.\footnote{Thus, $\vect i H_p(\Hd)$ inherits with the product $(\vect i x)(\vect i y)=\frac 1 2\vect i(x y+ yx)$.}
Thus, a natural choice for a maximal tripotent of $A_{2p}(\C)$ is $\vect i I_p$. The corresponding $\Omega$ is $\vect i\widetilde\Omega  $, where $\widetilde \Omega$ is the symmetric cone associated with $H_p(\Hd)$, that is, the cone of non-degenerate positive hermitian  $p\times p$ matrices over $\Hd$. 
We may then choose the frame $(\vect i e_1,\dots, \vect i e_p)$, where the $e_j$ are defined as for the domains of type $(I_{p,p})$, in which case  $\Delta_j(z)$ is the (complex) Pfaffian corresponding to the first $2j$ rows and columns of $z$, for every $z\in Z$.

We have thus dealt with the case in which $n$ is even. If, otherwise, $n=2p+1$ for some $p\Meg 1$, then we may choose a maximal tripotent $e=(\begin{smallmatrix} \vect i I_p & 0\\ 0 & 0 \end{smallmatrix})$, with the above identifications, so that $Z_1(e)$ may be canonically identified with $A_{2p}(\C)$, while $Z_{1/2}(e)$ may be identified with $M_{2p,1}(\C)$. It will then suffice to observe that
\[
\Phi(u,v)=\vect i \overline v \trasp u + u \trasp {\overline v} \vect i\in A_{2p}(\C)
\]
for every $u,v\in M_{2p, 1}(\C)$.

In particular, the rank of $D$ is $[\frac n 2]$, $a=4$, the genus of $D$ is $ 2(n-1)$ and $D$ is  of tube type if and only if $n$ is even.

\subsection{Domains of Type $(I\!I\!I_n)$}

In this case, $Z$ is the space  $S_n(\C)$ of symmetric complex $n\times n$ matrices, with triple product given by  $\{x,y,z\}=\frac 1 2 (x y^* z+z y^* x)$ and associated domain $\Set{z\in Z\colon \norm{z}<1}$, where $\norm{z}$ denotes the usual operator norm of $z\in Z$. 

One may then choose $e\coloneqq I_p$ as a maximal tripotent. Then, $Z=Z_1(e)$ and $A(e)=\Set{z\in Z\colon z=z^*  }$, that is, $A(e)$ is the Jordan algebra of real symmetric $n\times n$ matrices. Then, $\Omega$ is the cone of non-degenerate positive symmtric real  $n\times n$ matrices, and  
\[
\Cc(z)= i  (I_n-z)^{-1}(I_n+z).
\]
One may then choose a frame $(e_1,\dots,e_n)$ of $A(e)$ as in the case of the domains of type $(I_{n,n})$, so that $\Delta_j(z)$ becomes the complex minor corresponding to the first $j$ rows and columns of $z$, for every $z\in Z$ and for every $j=1,\dots, n$.

In particular, the rank of $D$ is $n$, $a=1$, the genus of $D$ is $n+1$, and $D$ is always of tube type.

\subsection{Domains of Type $(IV_n)$, $n\Meg 3$}

In this case, $Z$ is the  hermitification of the (formally real) Jordan algebra $A$ with underlying space  $\R\times \R\times \R^{n-2}$ and   product given by $(x,y,z)(x',y',z')= (x x'+\langle z, z'\rangle, y y'+\langle z,z'\rangle, [(x+y) z'+(x'+y')z]/2)$, where $\langle\,\cdot\,,\,\cdot\,\rangle$ denotes the standard scalar product on $\R^{n-2}$.
Observe that $A$ may be also interpreted as the space of symmetric matrices $\left(\begin{smallmatrix} x &z\\z & y \end{smallmatrix}\right)$ endowed with the symmetrization of the product 
\begin{equation}\label{eq:4}
	\Big(\begin{smallmatrix} x &z\\z & y \end{smallmatrix}\Big)\left(\begin{smallmatrix} x' &z'\\z' & y' \end{smallmatrix}\right)=\left(\begin{smallmatrix} x x'+\langle z,z'\rangle & x z'+y' z\\ x' z+y z' & y y'+\langle z,z'\rangle \end{smallmatrix}\right)
\end{equation}
It is then natural to choose the identity $(1,1,0)$ of $A$ as a maximal tripotent for $Z$. In this case, $\Omega$ is the Lorentz cone $\Set{(x,y,z)\in Z\colon  x,y>0, xy-\abs{z}^2>0}$. 

One may then choose a frame $(e_1,e_2)$ of $A=A(e)$ with $e_1=(1,0,0)$ and $e_2=(0,1,0)$, so that $\Delta_1(x,y,z)=x$ and $\Delta_2(x,y,z)=x y-\sum_{j=1}^{n-2} z_j^2$.

In particular, the rank of $D$ is $2$, $a=n-2$, the genus of $D$ is $n$, and $D$ is always of tube type.

\subsection{The Exceptional Domains $(V)$ and $(V\!I)$}

For the exceptional domain $(V)$, $Z$ is the space of $1\times 2$ matrices with values in the complexification $\Od_\C$ of the division algebra $\Od$ of Cayley octonions, with triple product given by $\{x,y,z\}=\frac{1}{2}(x(\overline y^* z)+z(\overline y^* x))$, where $\overline{\,\cdot\,}$ is the conjugation of $\Od_\C$ (considered as the complexification of $\Od$), while $y^*=\trasp{\tilde y}$, where $\tilde{\,\cdot\,} $ is the complexification of the conjugation of $\Od$. In this case, the rank of $D$ is $2$, $a=6$, the genus of $D$ is $12 $, and $D$ is not of tube type.

For the exceptional domain $(V\!I)$, $Z$ is the hermitification of the formally real Jordan algebra $H_3(\Od)$ of hermitian $3\times3$ matrices with values   $\Od$, with the product $(x,y)\mapsto\frac 12 (xy+yx)$. In this case, the rank of $D$ is $3$, $a=8$, the genus of $D$ is $ 18$, and $D$ is of tube type.

\section{Invariant Closed Subspaces  of $\Hol(D)$}\label{sec:2}

In this section, $D$ denotes an irreducible symmetric \emph{circular} (hence convex) bounded domain. We shall then assume that $D$ is the unit ball  in a simple positive hermitian Jordan triple system $Z$ (with respect to the `spectral norm' $z\mapsto \norm{D(z,z)}$, where $Z$ is endowed with the scalar product $(z,w)\mapsto \tr D(z,w)$).

Let $\widetilde G(D)$ be the universal covering group of the component of the identity $G_0(D)$ of the group $G(D)$ of biholomorphisms of $D$, and define a representation $\widetilde U_\lambda$ of $\widetilde G(D)$ in $\Hol(D)$, for every $\lambda\in \R$, so that
\[
\widetilde U_\lambda(\phi) f = (f\circ \phi^{-1}) (J \phi^{-1})^{\lambda/g}
\]
for every $\phi\in \widetilde G(D)$ and for every $f\in \Hol(D)$, where $g$ is the genus of $D$, $\phi$ acts on $D$ by means of the canonical mapping $\widetilde G(D)\to G_0(D)$, and $(J \phi^{-1})^{\lambda/g}= \ee^{(\lambda/g) \log J(\phi^{-1},\,\cdot\,)}$, where $\log J$ is the unique continuous function on the simply-connected space $\widetilde G(D)\times D$ such that $\log J(I,0)=0$ and $\ee^{\log J(\phi,z)}=\det_\C \phi'(z)$ for every $(\phi,z)\in \widetilde G(D)\times D$. 

We shall also consider the \emph{ray representations} (cf., e.g.,~\cite{Bargmann}) $U_\lambda\colon G(D)\to \Lin(\Hol(D))/\T$ defined by
\[
U_\lambda(\phi) f=(f\circ \phi^{-1}) (J\phi^{-1})^{\lambda/g}
\]
for every $\phi\in G(D)$ and for every $f\in \Hol(D)$. Notice that, unless $\lambda/g\in\Z$, $U_\lambda$ may \emph{not} be unambiguously defined as an automorphism of $\Hol(D)$, so that $U_\lambda$ does not induce an ordinary representation of $G(D)$ in $\Hol(D)$: this is the reason why we need to define $U_\lambda$ as a ray representation. The advantage of considering $U_\lambda$ as a ray representation and not as a projective one lies in the fact that  ray representations in a semi-Banach space may be bounded or isometric, contrary to projective representations.
Recall that $U_\lambda$ is (strongly) continuous as a ray representation if the mapping $G(D)\ni \phi\mapsto U_\lambda (\phi)f\in \Hol(D)/\T$ is continuous for every $f\in \Hol(D)$.
When $\lambda/g\in \Z$, we shall also identify  $U_\lambda$ with the corresponding ordinary representation of $G(D)$.

In this section we shall investigate the closed $\widetilde U_\lambda$-invariant subspaces of $\Hol(D)$. To begin with, we shall determine the closed $K$-invariant and $K_0$-invariant subspaces of $\Hol(D)$, where $K$ denotes the stabilizer of $0$ in $G(D)$ (so that $K$ is the group of automorphisms of $Z$) and $K_0$ its component of the identity (so that $K_0$ is the stabilizer of $0$ in $G_0(D)$. Notice that by $K$-invariance we means $K$-$U_0$-invariance (and analogously for $K_0$-invariance). Thus, $K$-invariance is equivalent (for a vector subspace) to $K$-$U_\lambda$-invariance for every $\lambda\in \R$, since $K$ acts by linear automorphisms on $D$ by Proposition~\ref{prop:2}.

Recall that, since $D$ (hence, $\Omega$) is irreuducible, the set $\N_\Omega$ of $\vect s$ for which $\Delta^{\vect s}$ is a polynomial on $Z$ is $\Set{\vect s\in \N^r\colon  s_1\Meg s_2\Meg\cdots\Meg s_r}$ (cf.~Section~\ref{sec:1}).

\begin{proposition}\label{prop:3}
	For every $\vect s\in \N_\Omega$, denote by $\Pc_{\vect s}$ the $K_0$-invariant subspace of $\Hol(D)$ generated by $\Delta^{\vect s}$. Then, $U_0$ induces pairwise inequivalent irreducible representations of $K$ in the $\Pc_{\vect s}$, and the space $\Pc$ of holomorphic polynomials on $Z$ is the  direct sum of the $\Pc_{\vect s}$.
\end{proposition}

\begin{proof}
	All the assertions, with $K$ replaced by $K_0$, follow from~\cite[Theorem 2.1]{FarautKoranyi2}. The fact that the $\Pc_{\vect s}$ are $K$-$U_0$-invariant follows from~\cite[Proposition 4.6]{Tubi}.
\end{proof}

We may then consider the character $\chi_{\vect s}$  of the irreducible representation of $K_0$ in $\Pc_{\vect s}$ induced  by $U_0$, i.e.,
\[
\chi_s\colon K_0\ni \phi\mapsto \tr U_0(\phi)\in \C,
\]
and the associated integral operator
\[
\pi_{\vect s}= U(\overline {\chi_{\vect s}})=\int_{K_0} \overline{\chi_{\vect s}(k)} U_0(k)\,\dd k.
\]
Then,~\cite[Theorems 27.44]{HewittRoss} shows that $\pi_{\vect s}\Pc_{\vect s'}=0$ if $\vect s\neq \vect s'$ and that $\pi_{\vect s}$ is the identity on $\Pc_{\vect s}$, for every $\vect s,\vect s'\in \N_\Omega$. Since the space $\Pc$ of holomorphic polynomials on $Z$ is dense in $\Hol(D)$ (as $D$ is circular and convex), this and Proposition~\ref{prop:3} imply that $\pi_{\vect s}$ is a projector of $\Hol(D)$ onto $\Pc_{\vect s}$, that $\pi_{\vect s} \pi_{\vect s'}=0$ when $\vect s\neq \vect s'$, and that $I=\sum_{\vect s\in \N_\Omega} \pi_{\vect s}$ pointwise on $\Hol(D)$ (cf., also,~\cite[Subsection 4.2]{Tubi}).

\begin{corollary}\label{cor:2}
	The mappings $V\mapsto V\cap \Pc$ and $W\mapsto \overline W$ (where $\overline W$ denotes the closure of $W$ in $\Hol(D)$) induce two  bijections which are inverse of one another between the set of closed $K_0$-invariant vector subspaces of $\Hol(D)$ and the set of $K_0$-invariant vector subspaces of $\Pc$. 
	
	In particular, the closed $K_0$-invariant subspace of $\Hol(D)$ are also $K$-invariant and are precisely the spaces of the form $\overline{\bigoplus_{\vect s\in N} \Pc_{\vect s}}$, as  $N$ runs through the set of subsets of $\N_\Omega$.
\end{corollary}

\begin{proof}
	Let $V$ be a closed $K_0$-invariant subspace of $\Hol(D)$, and observe that either $\pi_{\vect s}(V)=\Set{0}$ or $\pi_{\vect s}(V)=\Pc_{\vect s}$, by $K_0$-invariance, for every $\vect s\in \N_\Omega$. Set $N\coloneqq \Set{\vect s\in \N_\Omega\colon \pi_{\vect s}(V)=\Pc_{\vect s}}$, and observe that $\sum_{\vect s\in \N_\Omega} \pi_{\vect s}$ is necessarily a projector of $\Hol(D)$ onto $V$, so that $V=\overline{\bigoplus_{\vect s\in N} \Pc_{\vect s}}$. In particular, $V\cap \Pc=\bigoplus_{\vect s\in N} \Pc_{\vect s}$ and $V=\overline{V\cap \Pc}$.  The assertion follows.
\end{proof}

Now, we shall consider the $\widetilde U_\lambda$-invariant closed vector subspaces of $\Hol(D)$, which may be equivalently characterized as the closed $G_0(D)$-$U_\lambda$-invariant subspaces of $\Hol(D)$. In order to do that, we need the following definition.

\begin{definition}\label{def:2}
	For every $\vect s\in \N_\Omega$ and for every $\lambda\in\C$, define $q(\vect s, \lambda)$ as the order of $0$ of the function 
	\begin{equation}\label{eq:1}
		\lambda' \mapsto (\lambda')_{\vect s}\coloneqq  \prod_{j=1}^r \left( \lambda' - \frac 1 2 a(j-1)  \right)\cdots \left( \lambda'-\frac 1 2 a(j-1)+ s_j-1\right)
	\end{equation}
	at  $\lambda$ (cf.~Section~\ref{sec:1}), 	so that $q(\vect s,\lambda)$ is the number of $j\in \Set{1,\dots,r}$ such that $\frac 1 2 a(j-1)-\lambda$ is a positive integer $<s_j$. Set $q(\lambda)\coloneqq\max_{\vect s\in \N_\Omega} q(\vect s,\lambda)$.
\end{definition}

Cf.~\cite[Proposition 4.7]{Tubi} for a proof of the following result. The `if' part is a consequence of~\cite[Theorem 5.3]{FarautKoranyi2}. A seemingly incomplete proof of the `only if' part also appears in~\cite[Theorem 4.8 (ii)]{Arazy}.

\begin{proposition}\label{prop:8}
	Take $\lambda\in \R$. Then, a closed vector subspace $V$ of $\Hol(D)$ is $G_0(D)$-$U_\lambda$-invariant if and only if $V=\overline{\bigoplus_{q(\vect s,\lambda)\meg j} \Pc_{\vect s}}$ for some $j=-1,\dots,q(\lambda)$. In this case, $V$ is also $G(D)$-$U_\lambda$-invariant.
\end{proposition}

\section{Invariant Semi-Banach Spaces in $\Hol(D)$}\label{sec:3}

We keep the notation of Section~\ref{sec:2}.
In this section we shall consider a semi-Banach subspace $X$ of $\Hol(D)$ in which $\widetilde U_\lambda$ induces a (not necessarily continuous) bounded representation of $\widetilde G$. Without further conditions on $X$, these spaces may exhibit very pathological behaviours, as their topology may be completely unrelated to that of $\Hol(D)$. Consequently, we shall impose some further conditions that we are about to discuss.
We shall content ourselves to discuss some basic facts and ideas which will be of use when dealing with semi-Hilbert spaces or minimal and maximal spaces. 

When $X$ is Hausdorff, the natural condition to be imposed is that $X$ embeds continuously into $\Hol(D)$. Nonetheless, in~\cite{ArazyFisher2,Arazy} a seemingly different condition was imposed, namely: 
\begin{enumerate}
	\item[$(W\!I)$] $ U_0$ induces a continuous representation of $K_0$ in $X$ and for every Radon measure $\mi$ on $K_0$  and for every $f,g\in X$
	\[
	\langle U_0(\mi) f\vert g\rangle=\int_{K_0} \langle U_0(k) f\vert g \rangle\,\dd \mi(k),
	\]
	where $U_0(\mi) f(z)=\int_{K_0} U_0(\phi) f(z)\,\dd \mi(\phi)$ for every $z\in D$.
\end{enumerate}
Notice that the condition $(W\!I)$ is stated in a somewhat different form in~\cite{ArazyFisher2,Arazy}. Indeed, in~\cite{ArazyFisher2,Arazy} it is required that for every Radon measure $\mi$ on $K_0$ the operator $U_0(\mi)$ (defined as an operator of $\Hol(D)$) maps $X$ into itself continuously and that `the integral converges in the norm of $X$.' At a first glance, this latter condition may seem to mean that for every $f\in X$ the function $K_0\ni k\mapsto U_0(k)f\in X$ is measurable and that its integral coincides with $U_0(\mi)f$. Nonetheless, a more careful inspection of the proofs shows that the function $K_0\ni k\mapsto U_0(k)f\in X$ is actually assumed to be continuous (and this is \emph{not} a consequence of measurability unless $X$ is assumed to be separable).
Then, if we define $\mi$ as the Haar measure on $\T\subseteq K_0$, the operator $U_0(\mi)\colon f\mapsto  f(0) \chi_D$ maps $X$ continuously into itself. In particular, the mapping $f\mapsto f(0)$ is continuous on $X$ since $X$ is Hausdorff. Since $\widetilde U_\lambda$ is assumed to induce a bounded representation of $\widetilde G(D)$ in $X$, and since $\widetilde G(D)$ acts transitively on $D$, it is then clear that the mapping $f\mapsto f(z)$ is continuous on $X$ for every $z\in D$. By the closed graph theorem, it then follows that $X$ embeds continuously into $\Hol(D)$. In particular, if   $X$ is reflexive (for example, a Hilbert space), this implies that $X$ is separable (cf.~\cite[Proposition 2.14]{Rango1}). 

For non-Hausdorff spaces, various conditions have been considered. In~\cite{ArazyFisher,ArazyFisherPeetre}, in which case  $\lambda=0$ and $D$ is the unit disc in $\C$, the authors assume that $X$  embeds continuously into the Bloch space
\[
\Bc\coloneqq\Set{f\in \Hol(D)\colon \sup_{z\in D} (1-\abs{z}^2)\abs{f'(z)}<\infty},
\]
endowed with the corresponding seminorm, and that $ U_0$ induces a continuous (bounded) representation of $ G(D)$ in $X$. Nonetheless, it is unclear whether this condition prevents some   pathological spaces, such as a $U_0$-invariant Hausdorff subspace of the Dirichlet space
\[
\Dc\coloneqq\Set{f\in \Hol(D)\colon \int_{D}  \abs{f'(z)}^2\,\dd z<\infty},
\]
endowed with the corresponding seminorm (\emph{if existing}).
In~\cite{ArazyFisher2,Arazy}, condition $(W\!I)$ was again considered (and this stronger condition certainly prevents the preceding pathological spaces). Nonetheless, it is unclear whether this condition implies that the closure of $\Set{0}$ in $X$ is closed in $\Hol(D)$, as the conclusions of~\cite[Theorems 5.2]{Arazy} require (without proof). 
We shall therefore follow the approach of~\cite{Rango1,Tubi}, where $X$ is assumed to be `strongly decent' and `saturated'.

The notion of a `decent' linear functional on $X$  dates back to~\cite{RubelTimoney}, where the Bloch space $\Bc$ on the unit disc $D$ in $\C$ was characterized as the largest $U_0$-invariant Banach space of holomorphic functions which admits a decent continuous linear functional, that is, a non-zero  $L\in X'$  which extends to a continuous linear functional on $\Hol(D)$. Consequently, we say that $X$ is decent if it admits a decent continuous linear functional. Even though this notion is perfectly adapted to the study of maximal invariant spaces of holomorphic functions (even in greater generality, cf.~Theorem~\ref{theorem:3}), a somewhat stronger notion appears to be necessary in order to study invariant semi-Banach spaces of holomorphic functions.

\begin{definition}
	We say that a semi-Banach space $X$ of holomorphic functions on $D$ is strongly decent if the set $L$ of continuous linear functionals on $X$ which extend to continuous linear functionals on $\Hol(D)$ is dense in $X'$ (in the weak dual topology). Equivalently (cf.~\cite[Proposition 2.13]{Rango1}), if there is a ($\widetilde U_\lambda$-invariant) closed vector subspace $V$ of $\Hol(D)$ such that $X\cap V$ is the closure of $\Set{0}$ in $X$ and the mapping $X\to \Hol(D)/V$ is continuous. If, in addition, $V\subseteq X$, then $X$ is said to be saturated.
	
	We shall say that $X$ is ultradecent if the closure of $L\cap \overline B_{X'}(0,1)$ in the weak topology $\sigma(X',X)$ generates $X'$ (i.e., if $L$ has a non-zero characteristic as a vector subspace of $X'$).
\end{definition}

Refining the preceding argument (and taking into account Proposition~\ref{prop:7} and its proof), it is not hard to show that, given a semi-Banach space $X$ in $\Hol(D)$ in which $\widetilde U_\lambda$ induces a bounded representation of $\widetilde G$, if contition $(W\!I)$ holds, then $X$ is ultradecent (and saturated up to replacing the closure of $\Set{0}$ in $X$ with its closure in $\Hol(D)$); conversely, if $X$ is ultradecent and saturated, then it satisfies condition $(W\!I)$.

We observe explicitly that the ultradecency condition has the only purpose of allowing to reconstruct the continuous linear functionals on $X'$ as limits of sequences of decent (or trivial) linear functionals, and this is essential when discussing measurability matters. We do not know any examples of $\lambda$-invariant spaces (see Definition~\ref{def:1} below) which are not ultradecent.

\begin{definition}\label{def:1}
	Take $\lambda\in \R$. We shall say that a semi-Banach space $X\subseteq \Hol(D)$ is (strictly) $\lambda$-invariant if $X$ is strongly decent and saturated, and if $U_\lambda$ induces a (not necessarily continuous) bounded (resp.\ isometric) ray representation of $G(D)$ in $X$.
\end{definition}

Notice that, when $D$ is the unit disc in $\C$, then the Bloch space $\Bc(D)$ is strictly $0$-invariant, but $U_0$ does \emph{not} induce a continuous representation of $G(D)$ in $\Bc(D)$.

We shall now present some general results on the basic properties of $\lambda$-invariant spaces. Notice that, by our definitions, a $\lambda$-invariant Banach space embeds continuously into $\Hol(D)$.

\begin{proposition}\label{prop:7}
	Take $\lambda\in \R$ and a $\lambda$-invariant semi-Banach space $X\subseteq \Hol(D)$. Then, the following conditions are equivalent:
	\begin{enumerate}
		\item[\textnormal{(1)}] $\widetilde U_\lambda$ induces a continuous representation of $\widetilde G(D)$ in $X$;
		
		\item[\textnormal{(2)}] $  U_\lambda$ induces a continuous ray representation of $ G(D)$ in $X$;
		
		\item[\textnormal{(3)}] $X$ is ultradecent and $\Pc \cap \overline X$ is contained and dense in $X$, where $\overline X$ denotes the closure of $X$ in $\Hol(D)$;
		
		\item[\textnormal{(4)}] $X$ is ultradecent and separable.
	\end{enumerate}
	If the preceding conditions are satisfied, then the mappings $\delta_R\colon f \mapsto f(R\,\cdot\,)$, $R\in (0,1]$, induce a strongly continuous semigroup of endomorphisms of $X$.
\end{proposition}

We observe explicitly that, as the proof shows, the equivalence between (2)  (with $G(D)$ replaced by $\T$ and $\lambda=0$), (3),   and (4) holds even if we only assume that $X$ is simply $\T$-invariant. The proof is partially inspired by~\cite{ArazyFisher2}.

The relevance of condition (3) lies in the fact that the space $\Pc\cap \overline X$ must be one of the (finitely many) spaces classified in Proposition~\ref{prop:8}.

This result should be compared to~\cite[Theorems 1 and 2]{AlemanMas}. We observe explicitly that the spaces considered in~\cite{AlemanMas} are subject to the assumption that $\Hol(R D)\subseteq X$ continuously for every $R>1$, which allows to consider the semigroup of dilations by $f \mapsto f(R\,\cdot\,)$ for $R\in (0,1]$ without further assumptions. Then, when $D$ is the unit disc in $\C$, the continuity of the rotations is   equivalent to the left continuity at $1$ of the dilations, and both these properties are equivalent to the density of $\Pc$ in $X$. One implication follows from Proposition~\ref{prop:7}, whereas the other one follows from the fact that $\widetilde U_\lambda$ is clearly continuous on $\Pc$, endowed with the topology of $\Hol(R D)$ for some $R>1$.

\begin{proof}
	(1) $\implies$ (2). By (1), it is clear that $U_\lambda$ is continuous on $G_0(D)$, in particular at the identity. Then, $U_\lambda$ is continuous on $G(D)$.
	
	(2) $\implies$ (1). This is a consequence of~\cite[Theorem 3.2]{Bargmann}.
	
	(2) $\implies$ (3). Observe that, when restricted to $\T$, $U_\lambda$ coincides with $U_0$ \emph{as a ray representation}. Since $U_0$ may be interpreted as an ordinary representation,~\cite[Theorem 3.2]{Bargmann} shows that it is continuous on $\T$. Thus, we may consider the continuous linear operator
	\[
	\pi_k\colon \Hol(D)\ni f \mapsto \frac{1}{k!} f^{(k)}(0)(\,\cdot\,)^k= \int_\T \alpha^k U_0(\alpha) f \,\dd \alpha\in \Hol(D)
	\]
	for every $k\in \N$. We may also define a continuous linear operators
	\[
	\widetilde \pi_k\colon X\ni f \mapsto  \int_\T \alpha^k U_0(\alpha) f \,\dd \alpha\in X/V,
	\] 
	where $V$ denotes the closure of $\Set{0}$ in $X$.
	Observe that, if $V^\circ$ denotes the polar of $V$ in $\Hol(D)$, that is, the space of continuous linear functionals on $\Hol(D)$ which induce continuous linear functionals on $X$, then
	\[
	\langle \mi,\pi_k(f)\rangle=\int_\T \alpha^k \langle \mi, U_0(\alpha) f\rangle\,\dd \alpha=\langle \mi, \widetilde \pi_k(f)\rangle
	\]
	for every $\mi \in V^\circ $ and for every $f\in X$. Since $V$ is closed in $\Hol(D)$, this is sufficient to prove that $\pi_k(f)\in X$ and that $\widetilde \pi_k(f)=\pi_k(f)+V$. Consequently, $\pi_k$ induces a continuous linear projector of $X$ onto $X\cap \Pc_k$, where $\Pc_k$ denotes the space of homogeneous holomorphic polynomials on $Z$ of degree $k$. Now,  for every $f\in X$, the sequence
	\[
	L_k(f)\coloneqq \sum_{k=0}^m \left( 1- \frac{k}{m+1}\right) \pi_k(f)=\int_\T U_0(\alpha) f \sum_{k=-m}^m \left(1-\frac{\abs{k}}{m+1}\right)\alpha^k\,\dd \alpha
	\]
	converges to $f$ in $X$ thanks to the properties of the Fejér kernel. Thus, $\Pc \cap X$ is dense in $X$. 
	Let us now prove that $\Pc\cap X=\Pc \cap \overline X$.  Take $f\in \Pc\cap \overline X$. Then, there is a sequence $(f_j)$ of elements of $X$ which converges to $f$ in $\Hol(D)$, so that $\pi_k(f_j)\to \pi_k(f)$  in $\Pc_k$ for every $k\in\N$. Now, $\pi_k(f_j)\in \Pc_k\cap X$ for every $j,k\in \N$ and $\Pc_k\cap X$ is closed (being finite-dimensional) in $\Pc_k$, so that $\pi_k(f) \in \Pc_k\cap X$ for every $k\in\N$. Since $f=\sum_{k\meg N} \pi_k(f)$ for a suitable $N\in\N$, it then follows that $f\in \Pc \cap X$.

	Finally, let us  prove that $X$ is ultradecent. Set $C\coloneqq \sup_{\alpha\in \T} \norm{U_0(\alpha)}_{\Lin(X)}<\infty$, so that $\norm{L_k}_{\Lin(X)}\meg C$ for every $k\in\N$ by the properties of the Fejér kernel. 
	Then, take $x'\in \overline B_{X'}(0,1/C)$ and observe that, for every $k\in \N$, there is a  linear functional $x'_k\colon \bigoplus_{h\meg k}\Pc_h\to \C$ which extends the restriction of $x'$ to $\bigoplus_{h\meg k}\Pc_h\cap X$. Then, the linear functional $x'_k\circ L_k$ is continuous on $\Hol(D)$, coincides on $X$ with the continuous linear functional $x'\circ L_k$, and $\norm{x'\circ L_k}_{X'}\meg \norm{x'}_{X'}\norm{L_k}_{\Lin(X)}\meg 1$. In addition, $ x'_k\circ L_k=x'\circ L_k$ converges to $x'$ in the weak topology $\sigma(X',X)$. We have thus proved that, if $L$ denotes the space of continuous linear functionals on $X'$ which extend to continuous linear functionals on $\Hol(D)$, then $\overline B_{X'}(0,1/C)$ is contained in the closure of $L\cap \overline{B}_{X'}(0,1) $ in $\sigma(X',X)$, so that $X$ is ultradecent.
	
	(3)  $\implies$ (4). Obvious.

	(4) $\implies$ (1). By~\cite[Corollary 2 to Proposition 18 of Chapter VIII, \S\ 4, No.\ 6]{BourbakiInt2}, we may reduce to proving that the mapping 
	\[
	\widetilde G(D)\ni \phi\mapsto \langle x', \widetilde U_\lambda(\phi) f\rangle\in \C
	\]
	is measurable for every $f\in X$ and for every $x'\in \overline B_{X'}(0,1) $. Since $X$ is ultradecent, there is $r\in (0,1]$ such that $\overline B_{X'}(0,r)$ is contained in the closure of $ L\cap \overline B_{X'}(0,1)$ in the weak topology $\sigma(X',X)$, where $L$ denotes the set of continuous linear functionals on $X$ which extend to continuous linear functionals on $\Hol(D)$ (cf.~\cite[Exercise 21 of Chapter IV, \S\ 2]{BourbakiTVS}). Since $\overline B_{X'}(0,1/r)$ is compact and metrizable in the weak topology $\sigma(X',X)$, for every $x'\in \overline B_{X'}(0,1)$ we may find a sequence $(\mi_{x',j})$ of elements of $L\cap \overline B_{X'}(0,1/r) $ which converges to $x'$ in the weak topology $\sigma(X',X)$. Since clearly the mapping
	\[
	\widetilde G(D)\ni \phi\mapsto \langle \mi_{x',j}, \widetilde U_\lambda(\phi) f\rangle\in \C
	\]
	is continuous for every $j\in \N$, the assertion follows.
	
	Now, assume that conditions (1) to (4) hold, and observe that
	\[
	\delta_R f(z)=f(Rz)= \int_{\T} U_0(\alpha) f(z) \frac{1-R^2}{1-2R \Re \alpha+R^2}\,\dd \alpha
	\]
	for every $f\in \Hol(D)$, for every $R\in (0,1)$, and for every $z\in D$, so that the assertion follows from the properties of the Poisson kernel, arguing as before.
\end{proof}

\begin{corollary}\label{cor:1}
	Keep the hypotheses and the notation of Proposition~\ref{prop:7}. If the Hausdorff space associated with $X$ is reflexive, then  conditions \textnormal{(1)} to \textnormal{(4)} hold.
\end{corollary}

\begin{proof}
	Observe first that $X$ is separable thanks to~\cite[Proposition 2.14]{Rango1}. It will then suffice to prove that $X$ is ultradecent.  However, if $L$ is a vector subspace of $X'$ which is dense for the weak topology $\sigma(X',X)$, then $L$ is dense in $X'$ for the strong dual topology, so that $L\cap B_{X'}(0,1)$ is dense in $\overline B_{X'}(0,1)$ for the strong dual topology, in particular for the weak topology $\sigma(X',X)$. Since $X$ is strongly decent, it is clear that it is also ultradecent.
\end{proof}

Let us now briefly comment on~\cite[Theorem 7]{AlemanMas}. If $D$ is the unit disc in $\C$, $X$ is a $\lambda$-invariant \emph{Banach} space  for $\lambda>0$, and $\Hol(RD)\subseteq X$ for every $R>1$, then~\cite[Theorem 7]{AlemanMas} characterizes the $X$ for which the spaces
\[
D(X)=\Set{f'\colon f\in X} \qquad \text{and} \qquad A(X)=\Set{f\in \Hol(D)\colon f'\in X}
\]
are still $(\lambda+1)$- and  (for $\lambda>1$) $(\lambda-1)$-invariant, respectively, in terms of the properties of the Cesàro operator
\[
\Cc\colon f \mapsto \int_0^{\,\cdot\,} \frac{f(z)}{1-z}\,\dd z
\]
on $X$. More precisely,  $D(X)$ is $(\lambda+1)$-invariant if and only if $\Cc$ induces a continuous endomorphism of $X$, whereas $A(X)$ is $(\lambda-1)$-invariant (for $\lambda>1$) if and only if $I_X-\Cc$ is invertible.

Since it is unclear whether a precise analogue of the Cesàro operator $\Cc$ exists in more general settings,\footnote{Actually, some analogues of the Cesàro operator were considered in~\cite{NanaSehba,CPCarleson}. Nonetheless, these analogues were not defined as continuous endomorphisms of $\Hol(D)$, but rather as endomorphisms of suitable quotients of $\Hol(D)$ for which a canonical `representative' is not known.} this interesting result seem to be specific to the one-dimensional case. 

\section{Invariant semi-Hilbert Spaces in $\Hol(D)$}\label{sec:4}

We keep the notation of Sections~\ref{sec:2} and~\ref{sec:3}.
In this section we shall characterize the $\lambda$-invariant semi-Hilbert subspaces $H$ of $\Hol(D)$.

Denote by $\Kc$   the unweighted Bergman kernel, that is, the reproducing kernel of the Bergman space $A^2(D)=L^2(D)\cap \Hol(D)$, and observe that $A^2(D)$ is $U_g$-invariant with its norm, so that
\begin{equation}\label{eq:2}
	\Kc(z,z')=J\phi(z)\Kc(\phi(z),\phi(z')) \overline{J\phi(z')}
\end{equation}
for every $\phi\in G(D)$ and for every $z,z'\in D$ (cf., e.g.,~\cite[Proposition 1.4.12]{Krantz}). In addition, the sesquiholomorphic function $\Kc^{\lambda/g}$, which may be unambiguously defined as a holomorphic function of $\lambda$ (taking the value $1$ at $\lambda=0$) since $D\times D$ is convex, hence simply connected,  is the reproducing kernel of a reproducing kernel Hilbert space if and only if $\lambda$ belongs to the so-called Wallach set, that is,
\[
\Wc(\Omega)=\Set{j a/2\colon j=0,\dots, r-1 }\cup (a(r-1)/2,+\infty)
\]
(cf., e.g.,~\cite{VergneRossi}). We denote by $H_\lambda(D)$, or simply $H_\lambda$, the corresponding Hilbert space, so that $H_\lambda$ is the completion of the space of finite linear combinations of the $\Kc^{\lambda/g}(\,\cdot\,,z)$, as $z$ runs through $D$, endowed with the unique scalar product such that
\[
\langle \Kc^{\lambda/g}(\,\cdot\,,z)\vert \Kc^{\lambda/g}(\,\cdot\,,z')\rangle=\Kc^{\lambda/g}(z,z')
\]
for every $z,z'\in D$. 
In particular, $H_\lambda$ is strictly $\lambda$-invariant by the invariance properties~\eqref{eq:2} of $\Kc$.

Then, $H_g$ is the unweighted Bergman space $A^2(D)$ by definition. More generally, if we define the weighted Bergman space
\[
A^2_\lambda(D)\coloneqq \Set{f\in \Hol(D)\colon \int_D \abs{f(z)}^2 \Kc(z,z)^{-\lambda/g}\,\dd \nu_D(z)<\infty},
\]
where $\dd\nu_D(z)= \Kc(z,z)\,\dd z$ is the $U_g$-invariant measure on $D$, then $A^2_\lambda(D)=H_\lambda$ for $\lambda>g-1$ (with proportional norms) and $A^2_\lambda(D)=\Set{0}$ otherwise.

Now, observe that  Proposition~\ref{prop:3} shows that every strongly decent and saturated $K_0$-invariant semi-Hilbert subspace $H$ of $\Hol(D)$ (so that $U_0$ induces a \emph{continuous} representation of $K_0$ in $H$, cf.~Corollary~\ref{cor:1}) is the orthogonal direct sum of the $\Pc_{\vect s}$, $\vect s\in \N_\Omega$, that it contains (and is therefore also $K$-invariant, cf.~Proposition~\ref{prop:3}). In addition, every two $K$-invariant scalar products on  $\Pc_{\vect s}$ are proportional for every $\vect s\in \N_\Omega$, since the $\Pc_{\vect s}$ are $K$-$U_0$-irreducible. 
It is then customary to endow the space $\Pc$ of holomorphic polynomials on $Z$ with the Fischer scalar product
\[
\langle p\vert q \rangle_\Fc\coloneqq p(\overline{\nabla}) \overline q(0),
\]
where $\overline \nabla$ denotes the antiholomorphic gradient. In other words, if an orthonormal basis and the associated coordinates on $Z$ are fixed,  $p(z)=\sum_{\alpha} a_\alpha z^\alpha$, and $q(z)=\sum_{\beta} b_\beta z^\beta$ for every $z\in Z$, then $\langle p\vert q\rangle_\Fc=\sum_{\alpha} a_\alpha \overline{b_\alpha}$. In addition, the completion of $\Pc$ with respect to this scalar product is the Fock space (cf.~\cite[Propositon XI.1.1]{FarautKoranyi})
\[
\Fc\coloneqq \Set{f\in \Hol(D)\colon \frac{1}{\pi^{\dim Z}} \int_Z \abs{f(z)}^2 \ee^{-\abs{z}^2}\,\dd z<\infty }.
\]
In particular, since $K$ embeds as a subgroup of the unitary group of $Z$ by Proposition~\ref{prop:2}, it is clear that $\Fc$ is $K$-$U_0$-invariant with its scalar product, so that the same holds for $\Pc$. 
Then, the following result holds   (cf.~\cite[Corollary 3.7 and Theorem 3.8]{FarautKoranyi2}).
\begin{theorem}\label{prop:9}
	Take $\lambda\in \Wc(\Omega)$. Then, 
	\[
	\langle f\vert g \rangle_{H_\lambda}=\sum_{q(\vect s,\lambda)=0 } \frac{1}{(\lambda)_{\vect s}} \langle \pi_{\vect s}(f)\vert \pi_{\vect s}(g)\rangle_{\Fc}
	\]
	for every $f,g\in H_\lambda$, where  $(\lambda)_{\vect s}$ is defined in~\eqref{eq:1}.
\end{theorem}

In particular, this provides another description of $H_\lambda$, for every $\lambda\in \Wc(\Omega)$. 

\begin{remark}\label{oss:1}
	Take $\lambda,\lambda'\in \Wc(\Omega)$ with $\lambda\meg\lambda'$. Then, $H_\lambda \subseteq H_{\lambda'}$ continuously.
\end{remark}

\begin{proof}
	It suffices to observe that $(\lambda)_{\vect s}\meg (\lambda')_{\vect s}$ for every $\vect s$ with $q(\vect s,\lambda)=0$.
\end{proof}

One may then consider, on $\Pc$,  the analytic continuation of the the  hermitian form
\[
\langle p\vert q\rangle_{\lambda}=\sum_{\vect s\in \N_\Omega }\frac{1}{(\lambda)_{\vect s}} \langle \pi_{\vect s}(p)\vert \pi_{\vect s}(q)\rangle_{\Fc}
\]
which is defined for $\lambda\in\C$ such that $q(\lambda)=0$. This form is always $K$-$U_0$-invariant and $\dd \widetilde U_\lambda$-invariant by analytic continuation, but is not positive unless $\lambda\in \Wc(\Omega)$. Nonetheless, when $q(\lambda)>0$, it may induce suitable residual hermitian forms on $V_{\lambda,j}$, with radical $V_{\lambda, j-1}$, where $V_{\lambda,j}=\bigoplus_{q(\vect s,\lambda)\meg j}\Pc_{\vect s}$, $j=0,\dots,q(\lambda)$. It turns out that the resulting hermitian forms   are positive if and only if $\lambda\in \Wc(\Omega)$ and $j=0$, or $a(r-1)/2-\lambda\in\N$ and $j=q(\lambda)$ (cf.~\cite[Theorem 5.4]{FarautKoranyi}). We may then state the following definition.
\begin{definition}
	Take $\lambda\in a(r-1)/2-\N$. Then, we define
	\[
	\widetilde H_\lambda(D)\coloneqq \Set{f\in \Hol(D)\colon  \sum_{q(\vect s,\lambda)=q(\lambda)} \frac{1}{(\lambda)'_{\vect s}} 
		\norm{\pi_{\vect s}(f)}_\Fc^2<\infty  },
	\]
	or simply $\widetilde H_\lambda$,
	endowed with the corresponding  Hilbert seminorm, where 
	\[
	(\lambda)'_{\vect s}= \lim_{\lambda'\to \lambda} \abs*{\frac{(\lambda')_{\vect s}}{(\lambda'-\lambda)^{q(\lambda)}}}
	\]
	for every $\vect s\in \N_\Omega$ with $q(\vect s,\lambda)=q(\lambda)$.\footnote{We observe explicitly that, if we had not put the absolute value, the $(\lambda)'_{\vect s} $ would still have had the same sign, but could have been all negative.}
\end{definition}

Since the closure of $\Set{0}$ in $\widetilde H_\lambda$ is clearly $V_{q(\lambda)-1}$, which is closed in $\Hol(D)$, and since one may show that the canonical mapping $\widetilde H_\lambda \to \Hol(D)/V_{q(\lambda)-1}$ is continuous, it is clear that $\widetilde H_\lambda$ is strongly decent and saturated.
Conversely, we have the following result (cf.~\cite[\S\ 5]{Arazy} and~\cite[Theorem 4.8]{Tubi}). See also~\cite{ArazyFisher} and~\cite[Theorem 5]{AlemanMas} for alternative proofs when $D$ is the unit disc and $\lambda=0$ and $\lambda>0$, respectively.

\begin{theorem}
	Take $\lambda\in \R$. If $\lambda\in \Wc(\Omega)$, then $H_\lambda$ is a strictly $\lambda$-invariant Hilbert space contained in $\Hol(D)$. If $\lambda\in a(r-1)/2-\N$, then $\widetilde H_\lambda$  is a strictly $\lambda$-invariant semi-Hilbert space contained in $\Hol(D)$.
	
	Conversely, if $H$ is a non-trivial (resp.\ strictly) $\lambda$-invariant semi-Hilbert  space contained in $\Hol(D)$, then either $\lambda\in \Wc(\Omega)$ and $H=H_\lambda$ with equivalent (resp.\ proportional) norms, or $\lambda\in a(r-1)/2-  \N$ and $H=\widetilde H_\lambda$ with equivalent (resp.\ proportional) seminorms.
\end{theorem}

We have thus completely characterized the $\lambda$-invariant semi-Hilbert spaces contained in $\Hol(D)$. Nonetheless, their description is quite abstract and therefore not completely satisfactory. A number of equivalent descriptions have been developed both in this context and in the context of Siegel domains. We shall present (some of) these descriptions in the following sections.

Before we do that, we indicate how the scalar product of the $H_\lambda$, $\lambda>m/r-1$, may be used to characterize the dual of ultradecent and separable $\lambda$-invariant Banach spaces, following~\cite[Theorem 3]{AlemanMas}. Notice that, for $\lambda>m/r-1$, one has $q(\lambda)=0$, so that a non-trivial $\lambda$-invariant semi-Banach space must be Hausdorff.

\begin{proposition}\label{prop:12}
	Take $\lambda>m/r-1$, and a separable and ultradecent $\lambda$-invariant Banach space $X$ in $\Hol(D)$. Consider the mapping  $\Ff\colon X'\to \Hol(D)$ defined so that
	\[
	\Ff(L)(z)\coloneqq \overline{\langle L, \Kc^{\lambda/g}(\,\cdot\,,z)\rangle}
	\]
	for every $L\in X'$ and for every $z\in D$. Then, $\Ff $ is well defined and continuous, $\Ff (X')$ is $\lambda$-invariant with the induced topology, and the sesquilinear form
	\[
	X\times \Ff (X')\ni (f,h)\mapsto \lim_{R\to 1^-} \langle f(R\,\cdot\,)\vert h(R\,\cdot\,)\rangle_{H_\lambda}\in \C
	\]
	is well defined and induces an antilinear isomorphism of $\Ff (X')$ onto $X'$.
\end{proposition}

\begin{proof}
	Observe that $1=\Kc^{\lambda/g}(\,\cdot\,,0)\in X$ by Proposition~\ref{prop:7}, so that also 
	\[
	\Kc^{\lambda/g}(\,\cdot\,,\phi(0))=\overline{(J \phi)^{-\lambda/g} (0)}\widetilde U_\lambda(\phi)1
	\]
	belongs to $X$ for every $\phi\in \widetilde G(D)$. In addition, the mapping $\widetilde G(D)\ni \phi\mapsto \Kc^{\lambda/g}(\,\cdot\,,\phi(0))\in X$ is continuous by Proposition~\ref{prop:7}, so that using Morera's theorem and the preceding remarks we see that the mapping $D\ni z\mapsto \Kc^{\lambda/g}(\,\cdot\,,z)\in X$ is antiholomorphic. Thus, $\Ff $ is well defined and continuous. In particular, $\Ff (X')$ is a Banach space and embeds continuously into $\Hol(D)$. The $\lambda$-invariance of $\Ff (X')$ then follows from the preceding formulae, since $U_\lambda(\phi) \Ff (L)=c_\phi \Ff (L\circ U_\lambda(\phi^{-1}))$ for some $c_\phi\in \T$, for every $\phi\in G(D)$.
	
	Now, observe that, if $f\in H_\lambda$ and $(e_j)$ is an orthonormal basis of $H_\lambda$ consisting of homogeneous polynomials of degree $d_j$ (which is possible by Theorem~\ref{prop:9}), then $\Kc^{\lambda/g}(R\,\cdot\,,z)=\sum_j  e_j(R\,\cdot\,) \overline{e_j(z)}$, so that
	\[
	\langle f\vert \Kc^{\lambda/g}(R\,\cdot\,,z)\rangle_{H_\lambda}=\sum_{j} \langle f\vert e_j\rangle_{H_\lambda}R^{d_j} e_j(z) 	= \sum_j \langle f \vert e_j \rangle_{H_\lambda} e_j(R z)= f(Rz)
	\]
	for every $z\in D$ and for every $R\in (0,1)$. 
	In addition, 	observe that $f(R\,\cdot\,)\in H_\lambda$ for every $R\in (0,1)$ and for every $f\in \Hol(D)$ by Remark~\ref{oss:2} below. In particular, the mapping $X\ni f\mapsto f(R\,\cdot\,)\in H_\lambda$ is continuous for every $R\in (0,1)$.
	Therefore,   choosing $f=\Kc^{\lambda/g}(\,\cdot\,,z')$ for some $z'\in D$,  and $L\in X'$, 
	\[
	\langle f(R\,\cdot\,)\vert \Ff (L)(R\,\cdot\,)\rangle_{H_\lambda}=\overline{\Ff (L)(R^2\,\cdot\,)}=\langle L, f(R^2\,\cdot\,)\rangle.
	\]
	Since the set of $\Kc^{\lambda/g}(\,\cdot\,,z)$, $z\in D$, is total in $X$ by Proposition~\ref{prop:7}, the equality $\langle f(R\,\cdot\,)\vert \Ff (L)(R\,\cdot\,)\rangle_{H_\lambda}=\langle L, f(R^2\,\cdot\,)\rangle$ holds for every $f\in X$, so that
	\[
	\lim_{R\to 1^-} \langle f(R\,\cdot\,), \Ff (L)(R\,\cdot\,)\rangle_{H_\lambda}= \lim_{R\to 1^-} \langle L, f(R^2\,\cdot\,)\rangle=\langle L,f\rangle
	\]
	thanks to Proposition~\ref{prop:7}, whence the result.
\end{proof}

\section{Other Descriptions of the Spaces $H_\lambda$ and $\widetilde H_\lambda$}\label{sec:5}

We keep the notation of Sections~\ref{sec:2} to~\ref{sec:4}.
In addition, with reference to a fixed maximal tripotent $e$ of $Z$, we set $m\coloneqq \dim_\C Z_1(e)$ and $n\coloneqq \dim_\C Z_{1/2}(e)$, so that $m/r-1=a(r-1)/2$ and $g=(n+2 m)/r$.  Notice that $m,n,a,g$ do \emph{not} depend on the choice of $e$.
In addition, we define $\square\coloneqq \Delta(\nabla)$.

\subsection{The Case of Domains of Tube Type}

For  domains of tube type, we have the following intertwining formula (cf.~\cite[Theorem 6.4]{Arazy}).

\begin{theorem}\label{prop:10}
	Take $\lambda\in m/r-1-\N$ and assume that $D$ is of tube type (that is, $n=0$). Then, $\square^{m/r-\lambda}$ intertwines the representations $\widetilde U_\lambda$ and $\widetilde U_{g-\lambda}$ of $\widetilde G$ in $\Hol(D)$. In particular, 
	\[
	\widetilde H_\lambda=\Set{f\in \Hol(D)\colon \square^{m/r-\lambda}f\in H_{g-\lambda}}
	\]
	with a proportional seminorm.
\end{theorem}

When $g-\lambda>g-1$, that is, $\lambda<1$, this provides an integro-differential description of  $\widetilde H_\lambda$, since $H_{g-\lambda}=A^2_{g-\lambda}(D)$. As we shall see later, in the context of Siegel domains this provides an integro-differential description of all the spaces $\widetilde H_\lambda$, without any restrictions on $\lambda$. 

Another description of $\widetilde H_\lambda$ for all $\lambda\in m/r-1-\N$ is provided in~\cite[Theorem 12]{Arazy2}. Here, $H^2(D)$ denotes the Hardy space on $D$, which is the space $H_{(n+m)/r}$ with a proportional norm (cf.~Subsection~\ref{sec:5:4}).

\begin{theorem}
	Take $\lambda\in m/r-1-\N$ and assume that $D$ is of tube type (that is, $n=0$). Then,
	\[
	\langle f\vert g \rangle_{\widetilde H_\lambda}= \langle \Delta^{m/r-\lambda} \square^{m/r-\lambda} f \vert g \rangle_{H^2(D)}
	\]
	for every $f,g\in \widetilde H_\lambda$.
\end{theorem}

See also~\cite{Arazy2} for other integral descriptions of the $\widetilde H_\lambda$ in this context.

\subsection{The Case of the Unit Ball}

When $D$ is the unit ball in $\C^{n+1}$, one may provide some interpretations of $H_\lambda$, for $\lambda>0$, and of $\widetilde H_\lambda$, for $\lambda\in-\N$, which do not seem to have been extended to more general domains (which are not of tube type). These results should nonetheless be compared with those of Subsection~\ref{sec:5:3}. In order to facilitate this comparison, we shall adopt similar notation. We observe explicitly that, by Remark~\ref{oss:1} below, it is sufficient to provide such a description only for the $f$ which belong to $\Hol(R D)$ for some $R>1$.

Denote by $D^1_0$ the radial derivative on $\C^{n+1}$, so that $D^1_0 f(z)= z f'(z)$ for every $f\in \Hol(\C^{n+1})$ and for every $z\in \C^{n+1}$. Then, define 
\[
D^k_s\coloneqq (D_0^1+s)( D_{0}^1+s+1)\cdots (D_{1}^1+s+k-1).
\]
for every $s\in \R$ and for every $k\in \N$. 

Then, the following description of $H_\lambda$, $\lambda>0$, holds (cf.~\cite{Peetre3,Arazy3} and also~\cite[Theorem 7.1]{Arazy}). 

\begin{theorem}
	Take $\lambda>0$ and $f\in \Hol(R D)$ for some $R>1$. Then, 
	\[
	\norm{f}_{H_\lambda}^2=\abs{f(0)}^2+\frac{1}{n!} \int_{D} D_0^{n+1} f (z) \overline{f(z)} \frac{(1-\abs{z}^2)^{\lambda-1}}{ \abs{z}^{2(n+1)}}\,\dd z.
	\]
	If, in addition, $\lambda>1$, then
	\[
	\norm{f}_{H_\lambda}^2=\frac{\lambda-1}{n!} \int_{D} D_1^{n} f (z) \overline{f(z)} \frac{(1-\abs{z}^2)^{\lambda-2}}{ \abs{z}^{2n}}\,\dd z.
	\]
\end{theorem}

Notice that, when $\lambda>1$, $H_\lambda$ is a weighted Bergman space, so that a simpler description applies.

There is also the following description of $\widetilde H_\lambda$, which was first found in~\cite{Peloso} for $\lambda=0$, and then extended to all $\lambda\in-\N$ (cf.~\cite{Peetre3,Arazy3} and also~\cite[Theorem 7.2]{Arazy}).

\begin{theorem}
	Take $\lambda\in-\N$ and $f\in \Hol(R D)$ for some $R>1$. Then,
	\[
	\norm{f}_{\widetilde H_\lambda}^2= \frac{1}{(-\lambda)!}\left( \norm{f_{1-\lambda}}_\Fc^2+\frac{1}{n!} \int_D D_{\lambda-1}^{n+2-\lambda} f(z) \overline{f(z)} \frac{1}{\abs{z}^{2(n+2-\lambda)}} \,\dd z\right)
	\]
	and
	\[
	\norm{f}_{\widetilde H_\lambda}^2 =\frac{1}{(-\lambda)!n!} \langle D_{\lambda}^{n+1-\lambda} f \vert f \rangle_{H^2(D)}.
	\]
\end{theorem}

Recall that $\norm{\,\cdot\,}_\Fc$ denotes the Fischer norm (cf.~Section~\ref{sec:4}).

\subsection{Description by Means of Invariant Differential Operators}\label{sec:5:3}

Here we summarize some of the descriptions of the spaces $H_\lambda$ and $\widetilde H_\lambda$ provided in~\cite{Yan}.

To begin with, we need to introduce a family of $K$-invariant differential operators. For every $s\in \C$ and for every $k\in\N$, define $D^k_s$ as the unique ($K$-invariant) holomorphic differential operator (with polynomial coefficients) such that
\[
D_s^k p =(s \vect 1_r+\vect s)_{k \vect 1_r} p
\]
for every $p\in \Pc_{\vect s}$ and for every $\vect s\in \N_\Omega$ (cf.~\cite[Theorem 2.5]{Yan} and~\eqref{eq:1}).
In particular,
\[
D_s^k\coloneqq D_s^1 D_{s+1}^1\cdots D_{s+k-1}^1.
\]

Observe that, if $D$ is of tube type, then (cf.~\cite[Theorem 1.9]{Yan})
\begin{equation}\label{eq:3}
	D^k_s f=\Delta^{m/r-s} \square^k( \Delta^{s+k-m/r} f)
\end{equation}
for every $f\in \Hol(U)$ and for every simply-connected open subset $U$ of $\Delta^{-1}(\C\setminus \Set{0})$.\footnote{Here, we define $ \Delta^{m/r-\lambda} =\Delta^k/ \Delta^{\lambda+k-m/r} $. In other words, we define both $\Delta^{m/r-\lambda}$ and $\Delta^{\lambda+k-m/r}$ by means of the same holomorphic determination of $\log \Delta$ on $U$.}

We shall now describe the norms of the spaces $H_\lambda$ and $\widetilde H_\lambda$ on the space of holomorphic functions which are holomorphic on $R D$ for some $R>1$. As the following remark shows, this is sufficient to describe the spaces $H_\lambda$ and $\widetilde H_\lambda$ completely. Actually, it would suffice to determine the seminorm on the various $\Pc_{\vect s}$ (cf.~Proposition~\ref{prop:3}), but this Hardy-type description may provide additional insight.

\begin{remark}\label{oss:2}
	Take $\lambda\in \Wc(\Omega)$. For every $f$ in the closure $ \overline{H_\lambda}$ of $H_\lambda$ in $\Hol(D)$ and for every $R\in (0,1)$, $f(R\,\cdot\,)\in H_\lambda$. In addition, $f\in H_\lambda$ if and only if $ \sup_{R\in (0,1)} \norm{f(R\,\cdot\,)}_{H_\lambda}<\infty$, in which case $\norm{f}_{H_\lambda}=\lim\limits_{R\to 1^-} \norm{f(R\,\cdot\,)}_{H_\lambda}=\sup_{R\in (0,1)} \norm{f(R\,\cdot\,)}_{H_\lambda}$.
	
	An analogous assertion holds if $\lambda\in m/r-1-\N$ and $H_\lambda$ is replaced by $\widetilde H_\lambda$.
\end{remark}

Notice that   $H_\lambda$ and $\widetilde H_\lambda$ are dense in $\Hol(D)$ if (and only if) $\lambda>m/r-1$ and $\lambda\in m/r-1-\N$, respectively.

\begin{proof}
	Take $f\in \overline{H_\lambda}=\overline{\bigoplus_{q(\vect s,\lambda)=0}\Pc_{\vect s}}$, so that $f=\sum_{q(\vect s,\lambda)=0} f_{\vect s}$ in $\Hol(D)$, with $f_{\vect s}\in \Pc_{\vect s}$ for every $\vect s$ with $q(\vect s,\lambda)=0$. 
	Then, take $R\in (0,1)$ and let us prove that $f(R\,\cdot\,)\in H_\lambda$. Observe first that, for every $R'>0$, clearly $f(R'\,\cdot\,)$ belongs to the unweighted Bergman space $A^2(D)=H_g$, so that Theorem~\ref{prop:9} implies that 
	\[
	\sum_{q(\vect s,\lambda)=0}\frac{R'^{2\abs{\vect s}}}{(g)_{\vect s}} \norm{f_{\vect s}}_{\Fc}^2\meg \sum_{\vect s} \frac{1}{(g)_{\vect s}} \norm{f_{\vect s}(R'\,\cdot\,)}_{\Fc}^2=\norm{f(R'\,\cdot\,)}_{H_g}^2<\infty.
	\]
	Then, fix $R'\in (R,1)$, and observe that   there is $N\in\N$ such that
	\[
	\lambda+N-\frac 1 2 a(j-1)\Meg \frac{R}{R'}\Big(g+N-\frac 1 2 a(j-1)\Big)
	\]
	for every $j=1,\dots,r$. It then follows that
	\[
	\sum_{q(\vect s,\lambda)=0} \frac{R^{2\abs{\vect s}} \norm{f_{\vect s}}_{\Fc}^2}{(\lambda)_{\vect s}}\meg C\sum_{\vect s}\frac{R'^{2\abs{\vect s}} \norm{f_{\vect s}}_{\Fc}^2}{(g)_{\vect s}}<\infty,
	\]
	where $C\coloneqq \max_{q(\vect s,\lambda)=0, \vect s\meg N\vect 1_r} \frac{(g)_{\vect s}}{(\lambda)_{\vect s}}$. It then follows that $f(R\,\cdot\,)\in H_\lambda$ and that
	\[
	\norm{f(R\,\cdot\,)}_{H_\lambda}^2=\sum_{q(\vect s,\lambda)=0} \frac{1}{(\lambda)_{\vect s}} \norm{f_{\vect s}(R\,\cdot\,)}_{\Fc}^2=\sum_{q(\vect s,\lambda)=0} \frac{R^{2\abs{\vect s}}}{(\lambda)_{\vect s}} \norm{f_{\vect s}}_{\Fc}^2
	\]
	is an increasing function of $R\in (0,1)$, thanks to Theorem~\ref{prop:9} again. Since
	\[
	\sup_{R\in (0,1)}\norm{f(R\,\cdot\,)}_{H_\lambda}^2=\lim_{R\to 1^-}\norm{f(R\,\cdot\,)}_{H_\lambda}^2=\sum_{q(\vect s,\lambda)=0} \frac{1}{(\lambda)_{\vect s}} \norm{f_{\vect s}}_{\Fc}^2
	\]
	by monotone convergence, the first assertion follows. The second assertion is proved analogously. 
\end{proof}

The following result is contained in~\cite[Theorems 3.13 and 3.15]{Yan}.

\begin{theorem}
	Take $\lambda>m/r-1$ and $k\in \N$ such that $\lambda+k>g-1$. Then, there is a constant $c_{\lambda,k}\neq 0$ such that
	\[
	\norm{f}_{H_\lambda}^2= c_{\lambda,k} \int_D D^k_\lambda f(z) \overline{f (z)}\Kc(z,z)^{-(\lambda+k)/g}\,\dd \nu_D(z)=c_{\lambda,k}\langle D^k_\lambda f \vert f \rangle_{H_{\lambda+k}}
	\]
	for every $f\in  \bigcup_{R>1}\Hol(R D)$. 
	
	The same holds with $\lambda\in m/r-1-\N$ and $H_\lambda$ replaced by $\widetilde H_\lambda$ (and with a different constant $c_{\lambda,k}$).
\end{theorem}

Now, define $N\colon Z\to [0,\infty)$ as the unique $K$-invariant function on $Z$ which induces $\abs{\Delta}$ on $Z_1(e)$ (cf.~\cite[p.~21]{Yan}). Then,~\cite[Theorem 3.9]{Yan} implies the following result.

\begin{theorem}\label{prop:6}
	Take $\lambda\in m/r-1-\N$ and $k\in\N$ such that $k>(n+m)/r-1$ and $k\Meg m/r-\lambda$. Then, there is a constant $c_{\lambda,k}\neq 0$ such that
	\[
	\norm{f}_{\widetilde H_\lambda}^2= c_{\lambda,k} \int_D D^{m/r-\lambda}_{\lambda+n/r} f(z) \overline{D^k_\lambda f(z)} \frac{\Kc(z,z)^{-( m/r+ k)/g}}{N(z)^{2(m/r-\lambda)}}\,\dd\nu_D( z)
	\]
	for every $f\in \bigcup_{R>1}\Hol(R D)$.
\end{theorem}

When $D$ is of tube type and $k=m/r-\lambda$, the preceding expression becomes 
\begin{eqnarray*}
	\norm{f}_{\widetilde H_\lambda}^2&=& c_{\lambda,m/r-\lambda} \int_D \abs{D^{m/r-\lambda}_{\lambda} f(z)}^2 \frac{\Kc(z,z)^{-(g-\lambda)/g}}{\abs{\Delta^{m/r-\lambda}(z)}}\,\dd \nu_D(z)\\
	&=& c_{\lambda,m/r-\lambda} \norm{\square^{m/r-\lambda} f}_{H_{g-\lambda}}^2
\end{eqnarray*}
thanks to~\eqref{eq:3}; this latter formula may be considered as a particular case of Theorem~\ref{prop:10}.

We point out that~\cite[Theorem 3.9]{Yan} claims that, when $D$ is of tube type and $k>m/r-\lambda$, one has the equality
\[
(D^{m/r-\lambda} p) \overline{D^k_\lambda q} \frac{1}{\abs{\Delta}^{2(m/r-\lambda)}}=(\square^{m/r-\lambda} p) \overline{\Delta^{k+\lambda-m/r} \square^k q}
\]
for every two polynomials $p,q$ on $Z$ (and this would provide a simplified version of Theorem~\ref{prop:6} in this case). Nonetheless, this would imply the equality $D^k_\lambda =\Delta^k \square^k= D^k_{m/r-k} $, which is false.

\subsection{Generalized Hardy Spaces}\label{sec:5:4}

In~\cite{ArazyUpmeier2}, the classical identification of $H_{(n+m)/r}$ with the Hardy space $H^2(D)$ is extended to the spaces $H_{(n+m)/r+j a/2}$, $j=0,\dots, r-1$, which are identified with suitable generalized Hardy spaces associated with suitable measures on the various $G(D)$-orbits of the boundary of $D$ (the Hardy space corresponds to the smallest of these orbits, namely the \v Silov boundary of $D$). This identification was previously found in~\cite{VergneRossi}, where, however, $D$ was replaced by its unbounded realization as a Siegel domain of type II. 
In addition,~\cite{ArazyUpmeier2} also provides an analogous interpretation for the spaces $H_{j a/2}$, $j=0,\dots, r-1$, that is,  the spaces corresponding to the dicrete part of the Wallach set $\Wc(\Omega)$.

For some reasons, in~\cite{ArazyUpmeier2} the authors repeatedly identify the considered functions  with their `boundary values', even when these boundary values (in the usual sense of restricted admissible boundary values) do not exist, and also when these boundary values exist as restricted admissible boundary values, but not in the sense of unrestricted boundary values as it is often required in the statements. We shall discuss these issues more extensively when dealing with the unbounded realization of $D$ as a Siegel domain of type II, where these problems are more easily settled and highlighted. 
Here, we shall content ourselves with an amended version of~\cite[Theorem 6.8]{ArazyUpmeier2}, which describes the spaces $H_{(n+m)/r+j a/2}$, $j=0,\dots, r-1$, as generalized Hardy spaces.

To do so, we first need to observe that $\partial D$ contains exactly $r$ distinct $G(D)$-orbits $\partial_j D$, $j=0,\dots, r-1$, ordered in such a way that $\partial_j D\subseteq \overline{\partial_{j+1} D}$ for $j=0,\dots, r-2$. In particular, $\partial_0 D$ is the \v Silov boundary of $D$ (cf.~\cite[p.~419]{ArazyUpmeier2} and~\cite[Theorem 6.3]{Loos}). 

\begin{theorem}
	For every $j=0,\dots, r-1$ there is a $K$-invariant (positive) measure $ \mi_j$ on $\partial_j D$ such that
	\[
	\norm{f}_{H_{(n+m)/r+j a/2}}^2= \norm{f}_{L^2(\mi_j)}
	\]
	for every $f\in \bigcup_{R>1}\Hol(R D)$.
\end{theorem}

Notice that, when $j=0$, the action of $K$ on $\partial_0 D$ is transitive, so that $\mi_0$ is unique up to a constant. We refer the reader to~\cite{ArazyUpmeier2} for a description of $\mi_j$ for $j=1,\dots, r-1$. 

We refer the reader to~\cite{ArazyUpmeier2} for a description of $H_{j a/2}$, $j=0,\dots, r-1$, as it seems too technical for our purposes.

\section{Affinely Invariant Spaces on Siegel Domains}\label{sec:6}

We keep the notation of Sections~\ref{sec:2} to~\ref{sec:5}.
In this section, we shall work with the realization of $D$ as a Siegel domain, namely $\Dc$, associated with a fixed maximal tripotent $e$ in $Z$. We shall then set $E\coloneqq Z_{1/2}(e)$ and $F\coloneqq A(e)$, so that $Z_1(e)=F_\C$, $n=\dim_\C E$ and $m=\dim_\R F$. Consequently, we shall write $\overline z$ instead of $z^*=\{e,z,e\}$ (and $\Re z$, $\Im z$ instead of $(z+z^*)/2$, $(z-z^*)/(2i)$, respectively) for every $z\in F_\C$. Thus,
\[
\Dc=\Set{(\zeta,z)\in E\times F_\C\colon \Im z- \Phi(\zeta)\in \Omega},
\]
with the simplified notation $\Phi(\zeta)=\Phi(\zeta,\zeta)$.
We shall still denote by $\widetilde U_\lambda$ and $U_\lambda$ the (ordinary and ray) representations of $\widetilde G(\Dc)$ and $G(\Dc)$, respectively, in $\Hol(\Dc)$, defined as the corresponding representations of $\widetilde G(D)$ and $G(D)$. If we define  (cf.~Section~\ref{sec:1})
\[
\Cc_\lambda\colon \Hol(D)\ni f \mapsto (f \circ \Cc^{-1}) (J\Cc^{-1})^{\lambda/g}\in \Hol(\Dc),
\]
then $\Cc_\lambda$ is an isomorphism and intertwines the $\widetilde U_\lambda, U_\lambda$ defined on $\widetilde G(D), G(D)$ and the $\widetilde U_\lambda, U_\lambda$ defined on $\widetilde G(\Dc), G(\Dc)$. We may then define $H_\lambda(\Dc)\coloneqq \Cc_\lambda H_\lambda(D)$ for $\lambda\in \Wc(\Omega)$ and $\widetilde H_\lambda(\Dc)\coloneqq \Cc_\lambda \widetilde H_\lambda(D)$ for $\lambda\in m/r-1-\N$. Then, the $H_\lambda(\Dc)$ and the $\widetilde H_\lambda(\Dc)$ are the $\lambda$-invariant semi-Hilbert spaces in $\Hol(\Dc)$.
We observe explicitly that the unweighted Bergman kernel on $\Dc$ is
\[
((\zeta,z),(\zeta',z'))\mapsto c \Delta\left(\frac{z-\overline{z'}}{2i }-\Phi(\zeta,\zeta')  \right)
\]
for some $c>0$ (cf., e.g.,~\cite[Proposition 3.11]{CalziPeloso}).

In this context, the analogues of $K$ and $K_0$ no longer play a relevant role (as they do not have a reasonable description), and are therefore essentially replaced by the nilpotent group $\Nc=E\times F$, with product
\[
(\zeta,x)(\zeta',x')=(\zeta+\zeta',x+x'+2\Im \Phi(\zeta,\zeta')).
\]
The group $\Nc$ acts freely and faithfullly on $\Dc$ (and $E\times F_\C$) by affine biholomorphisms as follows:
\[
(\zeta,x)\cdot (\zeta',z')=(\zeta+\zeta',z'+x+i \Phi(\zeta)+2 i \Phi(\zeta',\zeta)).
\]
In particular, $\Nc$ may be identified with the \v Silov boundary $\Nc\cdot (0,0)$ of $\Dc$, namely $\Set{(\zeta,x+i \Phi(\zeta))\colon (\zeta,x)\in \Nc}$.

Notice that $\Nc$ is not a compact group, but rather a $2$-step nilpotent group (an abelian group, if $E=\Set{0}$). Even though this allows the use of Fourier techniques, it is no longer possible to decompose $\Hol(D)$ into the (closure of the) direct sum of countably many irreducible subrepresentations of some $U_\lambda$.

In this situation the group $\Aff(\Dc)$ (or simply $\Aff$) of affine automorphisms of $\Dc$ acts transitively on $\Dc$, and is generally easier to understand than the whole group $G(\Dc)$. We shall therefore try to describe the $\Aff$-invariant subspaces of $\Hol(\Dc)$. Notice that, when restricted to $\Aff$, $U_\lambda$ may be `corrected' in order to become an ordinary representation $\Uc_\lambda$ of $\Aff$, simply replacing $J\phi^{-1}$ with $\abs{J \phi^{-1}}$. Namely,
\[
\Uc_\lambda(\phi)f=(f\circ \phi^{-1}) \abs{J\phi^{-1}}^{\lambda/g}
\]
for every $\phi\in \Aff$ and for every $f\in \Hol(\Dc)$.

We shall then consider $\Uc_\lambda$-invariant spaces when $n=0$ (that is, when $\Dc$ is a tube domain) and when $m=1$ (that is, when $D$ is the unit ball in $\C^{n+1}$). Unfortunately, the current techniques do not seem to allow for a more general treatment.
We have nonetheless the following general result concerning the closed $\Aff$-invariant subspaces of $\Hol(\Dc)$ (cf.~\cite[Proposition 7.1]{Rango1}).
Here we denote by $GL(\Dc)$ the space of linear biholomorphisms of $\Dc$, so that 
\[
GL(\Dc)=\Set{A\times B_\C\colon A\in GL(E), B\in G(\Omega), B\Phi=\Phi(A\times A)},
\]
$\Aff$ is the semi-direct product of $\Nc$ and $GL(\Dc)$  (cf.~\cite[Proposition 2.1]{Murakami}), and $GL(\Dc)$ acts faithfully  by group automorphisms on $\Nc$.

\begin{proposition}\label{prop:4}
	Let $V$ be a closed $\Aff$-invariant closed subspace of $\Hol(\Dc)$. Then, $V$ is the closure of $V\cap \Pc$, where $\Pc$ denotes the space of holomorphic polynomials on $E\times F_\C$, and there is a $GL(\Dc)$-invariant space $\Ic$ of distributions supported at $(0,0)$ in $\Nc$ such that
	\[
	V=\Set{f\in \Hol(D)\colon \forall I\in \Ic\:\: f*I=0}.
	\]
\end{proposition}

Before we proceed, we need to introduce a few more objects related to the irreducible symmetric cone $\Omega$.

\subsection{Irreducible Symmetric Cones}\label{sec:6:2}

Recall that $F=A(e)$ is endowed with the structure of a (real) Jordan algebra,  and with a frame $e_1,\dots,e_r$ with respect to which the $\Delta_j$ and $\Delta^{\vect s}$ are defined. Then, define $\Delta^*_j$ and $(\Delta^*)^{\vect s}$ as the $\Delta_j$ and $\Delta^{\vect s}$, but this time using the frame $e_r,\dots, e_1$. In other words, $\Delta^*_j$ is the determinant polynomial associated with the Jordan algebra $A(e_{r-j+1}+\cdots+e_r)$, composed with the canonical projection $A(e)\to A(e_{r-j+1}+\cdots+e_r)$, and then $(\Delta^*)^{\vect s}=(\Delta_1^*)^{s_1-s_2}\cdots (\Delta_r^*)^{s_r}$. Then, it is known that there is a simply transitive subgroup $T_-$ of $G(\Omega)$ such that $\Delta^{\vect s}$ and $(\Delta^*)^{\sigma(\vect s)}$, when transferred to $T_-$ by means of the mappings $t\mapsto t e$ and $t\mapsto t^* e$, respectively, give rise to the same  \emph{character} of $T_-$, where
\[
\sigma(s_1,\dots, s_r)=\sigma(s_r,\dots, s_1)
\]
for every $(s_1,\dots, s_r)\in \C^r$. In addition, every character of $T_-$ may be obtained in this way.  
Instead of providing a general construction, we shall describe $T_-$ case by case as a group of lower triangular matrices (cf., e.g.,~\cite[Section 2.1]{CalziPeloso} for more details).  

\begin{itemize}
	\item $r\Meg 3$: in this case, $\Omega$ is the cone of positive  non-degenerate hermitian $r\times r$ matrices over $\R$ (domains of type $(I\!I\!I_r)$), $\C$ (domains of type $(I_{r,q})$), $\Hd$ (domains of type $(I\!I_{2r})$ and $(I\!I_{2r+1})$), or $\Od$ (only for $r=3$, corresponding to the exceptional domain of type $(V\!I)$), endowed with the scalar product $(x,y)\mapsto \Re\tr(x y)$. Then, choosing the frame $e_1,\dots, e_r$ with $e_j=(\delta_{p,j}\delta_{q,j})_{p,q=1,\dots,r}$, $\Delta_j(x)$ and $\Delta_j^*(x)$ become the principal minors of $x$ (over $\R,\C,\Hd$, respectively\footnote{Cf.~\cite[Appendix 2]{BourbakiA3} for a definition of the determinant of an invertible matrix with values in a non-commutative division ring. Notice that, in particular, the determinant of a matrix with values in $\Hd$ takes values in the abelianization of the multiplicative group $\Hd^*$ of $\Hd$, which may be canonically identified with $\R_+^*$. We set the determinant to be $0$ by definition if the matrix is not invertible.})  corresponding to the first and last $j$ rows and columns of $x\in \Omega$, respectively, at least in the case of matrices over $\R,\C,\Hd$.  Then, one may let $T_-$ be the group of lower triangular $r\times r$ matrices with strictly positive diagonal entries, endowed with the unique \emph{left} action by linear automorphisms such that $t\cdot e=t t^*$, where $t^*$ denotes the transpose conjugate of $t$. Except for the case of matrices over $\Od$, the action of $T_-$ is then given by
	\[
	t\cdot x=t x t^*.
	\]
	It is then easily verified, in this case, that the adjoint action (which we shall denote as a \emph{right} action) acts so that $e\cdot t= t^* t$, and this formula determines the adjoint action of $T_-$ in general. It is then readily verified that both $\Delta^{\vect s}$ and $(\Delta^*)^{\sigma(\vect s)}$ determine the same character $t\mapsto \prod_j t_{j,j}^{2 s_j}$ of $T_-$;
	
	\item $r=2$: in this case, $\Omega$ is the Lorentz cone $\Set{\left( \begin{smallmatrix} x &z \\ z & y\end{smallmatrix} \right)\colon x,y>0, z\in \R^{m-2}, \abs{z}^2<xy}$ (domains of type $(I_{2,q})$, $(IV_{m})$, and $(V)$), endowed with the scalar product $\left( \big( \begin{smallmatrix} x &z \\ z & y\end{smallmatrix} \big),\big( \begin{smallmatrix} x' &z' \\ z' & y'\end{smallmatrix} \big)\right)\mapsto xx'+yy'+2\langle z,z'\rangle$. With the choice of $e_1=\left( \begin{smallmatrix} 1&0 \\ 0 & 0\end{smallmatrix} \right)$ and $e_2=\left( \begin{smallmatrix}0&0 \\ 0 & 1\end{smallmatrix} \right)$, we then get $\Delta_1\left( \begin{smallmatrix} x &z \\ z & y\end{smallmatrix} \right)=x$, $\Delta_1^*\left( \begin{smallmatrix} x &z \\ z & y\end{smallmatrix} \right)=y$, and $\Delta_2\left( \begin{smallmatrix} x &z \\ z & y\end{smallmatrix} \right)=\Delta_2^*\left( \begin{smallmatrix} x &z \\ z & y\end{smallmatrix} \right)=yx-\abs{z}^2$. Then, one may let $T_-$ be the group of lower triangular matrices with strictly positive diagonal entries, which acts on $\Omega$ on the left so that
	\[
	t\cdot x= (t x) t^*=t (x t^*),
	\] 
	where the product defined as in~\eqref{eq:4}. The ajdoint (right) action is then given by $x\cdot t=(t^* x) t = t^* (x t )$.  It is then readily verified that both $\Delta^{\vect s}$ and $(\Delta^*)^{\sigma(\vect s)}$ determine the same character $t\mapsto \prod_j t_{j,j}^{2 s_j}$ of $T_-$;
	
	\item $r=1$: then, $\Omega=\R_+^*$ (domains of type $(I_{1,q})$), endowed with the usual scalar product. In this case, one considers the frame $e=1$, so that $\Delta_1=\Delta_1^*$ is the identity on $\Omega$, and one may choose $T_-=G(\Omega)$.
\end{itemize}

Further, let us observe that $T_-$ acts (on the left) also on $E$ in such a way that the mapping $(\zeta,z)\mapsto (t\cdot \zeta, t\cdot z)$ induces a biholomorphism of $\Dc $ (that is, $\Phi(t\cdot \zeta)=t\cdot \Phi(\zeta))$). For domains of type $(I_{p,q})$ with the above description, it suffices to set $A\cdot t=t A$. For domains of type $(I\!I_{2n+1})$ and for the exceptional domain of type $(V)$ the situation is more intricate and the existence of the mentioned action follows from the general theory (cf., e.g.,~\cite[p.~14--15]{Kaneyuki}).

In addition, there is a holomorphic family $(I^{\vect s})_{\vect s\in \C^r}$ of tempered distributions on $F$, supported in $\overline{\Omega}$, such that $\Lc I^{\vect s}=(\Delta^*)^{-\sigma(\vect s)}$ for every $\vect s\in \C^r$ (cf., e.g.,~\cite[Lemma 2.26 and Proposition 2.28]{CalziPeloso}). In addition, $I^{-\vect s}$ is supported at $\Set{0}$ if and only if $(\Delta^*)^{\sigma(\vect s)}$ is a polynomial, that is, if and only if $\vect s\in \N^r$ and $s_1\meg \cdots\meg s_r$. We shall also denote by $\N_{\Omega}^*$ the set of these $\vect s\in \N^r$, so that $\N_\Omega^*=\sigma(\N_\Omega)$. 

We shall also simply write $I^{s}$ instead of $I^{s\vect 1_r}$ for $s\in \C$, so that $\square^k f= f* I^{-k}$ for every $k\in\N$ and for every $f\in \Hol(\Dc)$.

\subsection{The Spaces $H_\lambda(\Dc)$}

Let us now describe  the spaces  $H_\lambda(\Dc)$  for $\lambda\in \Wc(\Omega)$. Define
\[
A^2_\lambda(\Dc)=\Set{f\in \Hol(\Dc)\colon \int_{\Dc} \abs{f(\zeta,z)}^2 \Delta^{\lambda}(\Im z-\Phi(\zeta))\,\dd \nu_\Dc(\zeta,z)<\infty }
\]
endowed with the corresponding Hilbert norm, where $\dd\nu_\Dc(\zeta,z)=\Delta^{-g}(\Im z-\Phi(\zeta))\,\dd (\zeta,z)$ is the $G(\Dc)$-invariant measure. Then, $H_\lambda(\Dc)=A^2_\lambda(\Dc)$ (with proportional norms) when $\lambda>g-1$, whereas $A^2_\lambda(\Dc)=\Set{0}$ otherwise.

In the general case, we may provide a Fourier-type description of the $H_\lambda(\Dc)$. In order to do that, take $\tau\in \overline{\Omega}$, and denote by $\Rc_\tau$ the radical of the bilinear form $\langle \tau, \Im \Phi\rangle$. Then, $\Nc/\ker \tau$ is the direct sum of  the abelian group $\Rc_\tau$ and the Heisenberg group ($\R$, if $E=\Set{0}$) $(E\ominus \Rc_\tau)\oplus (F/\ker \tau)$, where $E\ominus \Rc_\tau$ denotes the orthogonal complement of $\Rc_\tau$ in $E$. Then, the Stone--Von Neumann theorem (cf.~\cite[Theorem 1.50]{Folland}) shows that there is (up to unitary equivalence) a unique  irreducible continuous unitary representation $\pi_\tau$ of $\Nc$ into a Hilbert space $\Hs_\tau$ such that $\pi_\tau(\zeta,x)=\ee^{-i \langle \tau,x\rangle}$ for every $(\zeta,x)\in \Nc$. One may then choose $\Hs_\tau=\Hol(E\ominus \Rc_\tau)\cap L^2(\ee^{-2 \langle \tau, \Phi\rangle}\cdot \Hc^{2(n-d_\tau)})$, where $d_\tau=\dim_\C \Rc_\tau$ and $\Hc^k$ denotes the $k$-dimensional Hausdorff measure, and set
\[
\pi_\tau(\zeta+\zeta',x) \psi(\omega)=\ee^{\langle \tau, 2 \Phi(\omega,\zeta)-\Phi(\zeta)-i  x\rangle} \psi(\omega-\zeta)
\]
for every $\zeta,\omega\in E\ominus \Rc_\tau$, for every $\zeta'\in \Rc_\tau$, for every $x\in F$, and for every $\psi\in \Hs_\tau$ (cf., e.g.,~\cite[\S\ 2]{PWS}). Let us now define the class of (say, Borel) measurable vector fields $(v_\tau)\in \prod_{\tau\in \overline\Omega} \Hs_\tau$: we say that $(v_\tau)$ is (Borel) measurable if the mapping  $GL(\Dc)\ni A\times B\mapsto \Us_{A\times B}^{-1} v_{\trasp B \tau} \in \Hs_\tau$ is (Borel) measurable for every $\tau\in \overline \Omega$,\footnote{Notice that the action $GL(\Dc)\ni A\times B\mapsto \trasp B\in G(\Omega)$ of $GL(\Dc)$ on $\Omega$ has only $r$ distinct orbits (cf., e.g.,~\cite[\S\ 5.1]{VergneRossi}), so that we are only imposing $r$ (and not uncountably many) conditions.} where $\Us_{A\times B}\colon \Hs_\tau\to \Hs_{\trasp B \tau}$ is the operator defined by
\[
\Us_{A\times B} \psi\coloneqq \abs{{\det}_\C A'} (\psi\circ A'),
\]
where $A'\colon E\ominus \Rc_\tau\to E\ominus \Rc_{\trasp B \tau}$ is the map induced by $A$, and is unitary and intertwines $\pi_\tau\circ (A\times B)$ and $\pi_{\trasp B \tau}$ (cf.~\cite[Subsection 2.7]{Tubi}). For every $\tau\in \overline \Omega$, define $P_{\tau,0}$ as the orthogonal projector of $\Hs_\tau$ onto the space of constant functions on $E\times \Rc_\tau$.

If $\lambda\in \Wc(\Omega)$, then $I^{\lambda}$ is a positive measure on $\overline \Omega$, so that we may consider the direct integral $2^\lambda \int_{\overline \Omega} \Lin^2(\Hs_\tau) P_{\tau,0} \, \dd I^{\lambda}(\tau)$, defined as
\[
\Set{(v_\tau)\in \prod_{\tau\in \overline \Omega} \Lin(\Hs_\tau)P_{\tau,0}\colon (v_\tau) \text{ measurable}, 2^\lambda \int_{\overline \Omega} \norm{v_\tau}^2_{\Lin^2(\Hs_\tau)}  \, \dd I^{\lambda }(\tau)<\infty}
\]
(modulo negligible vector fields), endowed with the corresponding (Hilbert) norm. Here, $\Lin^2(\Hs_\tau)$ denotes the space of Hilbert-Schmidt endomorphisms of $\Hs_\tau$, for every $\tau\in \overline \Omega$, and $(v_\tau)\in\prod_{\tau\in \overline \Omega} \Lin(\Hs_\tau) $ is (Borel) measurable if and only if $(v_\tau w_\tau)$ is (Borel) measurable for every (Borel) measurable $(w_\tau)\in \prod_{\tau\in \overline \Omega} \Hs_\tau$. Then, we have the following result (cf.~\cite{VergneRossi,Ishi4,Tubi}).

\begin{proposition}
	The linear mapping $\Pc_\lambda\colon 2^\lambda \int_{\overline \Omega} \Lin^2(\Hs_\tau) P_{\tau,0} \, \dd I^{\lambda}(\tau)\to H_\lambda(\Dc)$ defined by
	\[
	\Pc_\lambda(v)(\zeta,z)=2^\lambda\int_{\overline \Omega} \tr(v_{\tau}\pi_{\tau}) \ee^{-\langle \tau, \Im z-\Phi(\zeta)\rangle}\,\dd I^{\lambda }(\tau)
	\]
	for every $v=(v_\tau)\in 2^\lambda \int_{\overline \Omega} \Lin^2(\Hs_\tau) P_{\tau,0} \, \dd I^{\lambda}$ and for every $(\zeta,z)\in \Dc$, is an isometric isomorphism.
\end{proposition}

Thus, $\Pc_\lambda$ provides a Fourier-type description of $H_\lambda(\Dc)$ for every $\lambda\in\Wc(\Omega)$.

When $\lambda>m/r-1$, $I^{\lambda}=c_\lambda\Delta^\lambda \cdot \nu_\Omega$ for a suitable $c_\lambda>0$ (cf.~\cite[Proposition 2.28]{CalziPeloso}), so that $I^{\lambda}$ is absolutely continuous with respect to the Plancherel measure $c\abs{\Delta}^{b}\cdot \Hc^m$ (cf.~\cite[Corollary 1.17 and Proposition 2.30]{CalziPeloso}), and $\Pc_\lambda$ is actually the inverse Fourier transform of $c'\Delta^{\lambda-b} v$ for suitable constants $c,c'\neq 0$. In this case, it is possible to define a suitable Besov space $B_\lambda$ on $\Nc$ of type $L^2$ so that, for every $f\in H_\lambda(\Dc)$, the functions $f_h\colon (\zeta,x)\mapsto f(\zeta,x+i \Phi(\zeta))$ converge  to some $f_0$ in $B_\lambda$, and the mapping $f\to f_0$ induces an isomorphism of $H_\lambda(\Dc)$ onto $B_\lambda$. We refer the reader to~\cite[Chapters 4 and 5]{CalziPeloso} for a more thorough description of the spaces $B_\lambda$ and the corresponding extension operator $B_\lambda\to H_\lambda(\Dc)$ (inverse to the boundary values operator $f\mapsto f_0$). Here, we shall only remark that convergence in $B_\lambda$ does not, in general, imply pointwise convergence almost everywhere. In fact, $B_\lambda$ is not a space of functions, in general, but embeds canonically in a quotient of the space of tempered distributions on $\Nc$.
For example, when $n=0$ one may define $B_\lambda$ as $\Fc^{-1}(L^2(\Delta^\lambda\cdot \nu_\Omega))$, where $\Fc$ denotes the Fourier transform on $F$, so that the preceding remarks become more apparent in this case. The general case requires a more delicate approach, since the non-commutative Fourier transform does not have a clear description for general distributions.

Notice that, when $\lambda>m/r-1$, the norm of the space $H_\lambda(\Dc)$ may be described efficiently by means of the operators $\square^k$ and the integral norms of the weighted Bergman spaces (cf.~\cite[Corollary 3.7 and Proposition 4.2]{Tubi}, and argue as in the proof of~\cite[Proposition 2.8]{Rango1}). Notice, though, that this description fails to determine $H_\lambda(\Dc)$ as a vector space in general. 

\begin{theorem}
	Take $\lambda>m/r-1$ and $k\in \N$. Then, $\square^k$ induces a multiple of an isometry from $H_\lambda(\Dc)$ onto $H_{\lambda+2k}(\Dc)$. In particular, if $\lambda+ 2 k>g-1$, then there is a constant $c>0$ such that
	\[
	\norm{f}^2_{H_\lambda(\Dc)}=c \int_{\Dc} \abs{\square^k f(\zeta,z)}^2 \Delta^{\lambda+ 2k}(\Im z-\Phi(\zeta))\,\dd \nu_\Dc(\zeta,z)
	\]
	for every $f\in H_\lambda(\Dc)$. If $\lambda>2(m/r-1)$, then 
	\[
	H_\lambda(\Dc)=\Set{f\in \Hol(\Dc)\colon \sup_{(\zeta,z)\in \Dc} \Delta^{\lambda/2}(\Im z-\Phi(\zeta)) \abs{f(\zeta,z)}<\infty, \square^k f\in H_{\lambda+2 k} }.
	\]
\end{theorem}

When $\lambda\in (n+m)/r+\Wc(\Omega)$, we may further describe $H_\lambda(\Dc)$ as a suitable `Hardy-type' space (cf.~\cite[3.6.2]{CalziPeloso}). 

\begin{proposition}
	Take $\lambda\in  (n+m)/r+\Wc(\Omega)$. Then, $H_\lambda(\Dc)$ is the space of $f\in \Hol(\Dc)$ such that
	\[
	\norm{f}_{H_\lambda(\Dc)}^2=\sup_{h\in \Omega}\int_{\overline \Omega}  \norm{f_{h+h'} }^2_{L^2(\Nc)}\,\dd I^{\lambda-(n+m)/r}(h')
	\]
	is finite, endowed with a multiple of the corresponding norm.
\end{proposition}

Notice that, if $\lambda>g-1$, then $I^{\lambda-(n+m)/r}=c_\lambda \Delta^{\lambda-g}\cdot \Hc^m$ for a suitable $c_\lambda>0$ (cf.~\cite[Proposition 2.28]{CalziPeloso}), so that this  provides a more cumbersome description of the weighted Bergman spaces $A_\lambda^2(\Dc)$. If, otherwise, $\lambda=(n+m)/r $, then $I^{0}=\delta_0$, so that $H_{(n+m)/r}(\Dc)$ is the ordinary Hardy space.

As shown in~\cite[Theorem 1.2]{Garrigos2} (in the case of tube domains, that is, $n=0$), all the `generalized Hardy spaces' $H_{m/r+j a/2}(\Dc)$ have suitable boundary values by which they are determined. More explicitly, if $f\in H_{m/r+j a/2}(\Dc)$, then $f_h$ has a limit $f_0$ in $L^2(\Hc^m\otimes I^{j a/2})$ for $h\to 0$, $h\in \Omega$. Nonetheless, if one wishes to reconstruct $f_0$ as the pointwise limit almost everywhere of the $f_h$, then one has to take $h$ in a cone $\Omega_0$ such that $\overline {\Omega_0}\setminus \Set{0}\subseteq\Omega$ (`restricted convergence'). One may also consider more general forms of almost everywhere convergence which generalize non-tangential convergence; this more general kind of convergence in usually called `admissible convergence'. Cf.~\cite{Garrigos2} for more information on generalized Hardy spaces on tube domains, and also~\cite{Koranyi} and the references therein for a discussion of admissible and restricted convergence in a broader context.

\subsection{The Case $n=0$}\label{sec:6:1}

In this case, $\Dc=F+i \Omega$, so that $\Aff$ becomes the semi-direct product of $\Nc$ and the group $G(\Omega)$ of linear automorphisms of $\Omega$ (acting on $F_\C$ by complexification), thanks to~\cite[Proposition 2.1]{Murakami}. 
The spaces $\Pc_{\vect s}$ described in Section~\ref{sec:2} naturally arise also in this situation, namely as the $G(\Omega)$-invariant (or, equivalently, $G_0(\Omega)$-invariant) spaces generated by $\Delta^{\vect s}$, $\vect s\in \N_\Omega$, where $G_0(\Omega)$ denotes the identity component of $G(\Omega)$, cf.~\cite[Theorem XI.2.4]{FarautKoranyi} and~\cite[Proposition 3.16]{Tubi}. 
Thus, Proposition~\ref{prop:4} shows that the closed $\Aff$-invariant subspaces of $\Hol(\Dc)$ are the closures of $\bigoplus_{\vect s\in N} \Pc_{\vect s}$ for suitable subsets $N$ of $\N_\Omega$. 

The dual description is somewhat more efficient. Denote by $\Dc_{\vect s}$ the set of differential operators of the form $f\mapsto f * \Fc^{-1}(p(-i  \,\cdot\,))$, where $\Fc$ denotes the Fourier transform, $p\in \Pc_{\sigma(\vect s)}$, and  $\sigma(s_1,\dots,s_r)=\sigma(s_r,\cdots,s_1)$.  Then, we have the following result (cf.~\cite[Corollary 3.20]{Tubi}):

\begin{proposition}
	A closed vector subspace $V$ of $\Hol(\Dc)$ is $\Aff$-invariant if and only if $V=\bigcap_{\vect s\in N} \ker \Dc_{\vect s}$ for some $N\subseteq \N_\Omega^*$. In addition, $\Dc_{\vect s}$ is the $G(\Omega)$-invariant vector space of differential operators on $\Dc$ generated by $f\mapsto f*I^{-\vect s}$.
\end{proposition}

Concerning $\Uc_\lambda$-invariant semi-Hilbert spaces, we have the following result (cf.~\cite[Theorem 3.23]{Tubi}).

\begin{theorem}	\label{theorem:1}
	Take $\lambda\in \R$ and let $H$ be a non-trivial strongly decent, saturated  semi-Hilbert subspace of $\Hol(\Dc)$ such that  $U_\lambda$ induces a bounded (resp.\ isometric) representation of $\Aff$ in $H$. Then, there are $k\in\N$ and a closed $\Aff$-invariant subspace of $\Hol(\Dc)$ such that $\lambda+2 k\in \Wc(\Omega)$ and $H\subseteq H_{\lambda,k}+V$ with an equivalent (resp.\ proportional) seminorm, where $H_{\lambda,k}=\Set{f\in \Hol(\Dc)\colon \square^k f\in H_{\lambda+2 k}(\Dc)}$, endowed with the corresponding seminorm, $V$ is endowed with the $0$ seminorm, and $H_{\lambda,k}\cap V=\ker \square^k$.
\end{theorem}

Concerning the $G(\Dc)$-$U_\lambda$-invariant spaces, we have the following results, which translate the analogous results from the bounded case to the present setting (cf.~\cite[Theorem 4.3 and Propositions 4.4 and 4.5]{Tubi}).

\begin{proposition}\label{prop:5}
	Take $\lambda\in \R$. A closed vector subspace $V$ of $\Hol(\Dc)$ is $G(\Dc)$-$U_\lambda$-invariant if and only if  $V=\Hol(\Dc)$ or $V=\bigcap_{\vect s\in N_{\lambda,k}}\ker \Dc_{\vect s}$ for some $k\in \Set{1,\dots, r}$ such that $\frac 1 2 a (k-1)-\lambda\in\N$, where 
	\[
	N_{\lambda,k}=\Set{\vect s\in \N_\Omega^* \colon  s_{\lambda,k,r-k+1}=\cdots=s_{\lambda,k,r}=\frac 1 2 a (k-1)-\lambda+1}.
	\]
\end{proposition}

\begin{proposition}\label{prop:20}
	Take $\lambda\in \R$. If $\frac m r -1-\lambda\in \N$, then $\widetilde H_\lambda(\Dc)=H_{\lambda,m/r-\lambda}$   (with a proportional seminorm). In addition, $\square^{m/r-\lambda}$ intertwines $U_\lambda$ and $U_{2 m/r-\lambda}$ (even as ray representations of $G(\Dc)$ in $\Hol(\Dc)$).
\end{proposition}

\subsection{The Case $r=1$}

In this case, $F=\R$, $E=\C^n$ and we may assume that $\Phi$ is the standard hermitian scalar product on $\C^n$. Then, $\Dc=\Set{(\zeta,z)\in \C^n\times \C\colon \Im z>\abs{\zeta}^2}$. Notice that $\Wc(\R_+^*)=[0,+\infty)$.

For simplicity, we shall denote by $\Pc_k(\C^n)$ the space of homogeneous holomorphic polynomials on $\C^n$ of degree $k$, and by $\Pc^k(\C)$ the space of holomorphic polynomials on $\C$ of degree $<k$ (in this latter case, we also allow $k=\infty$). Then, we have the following results (cf.~\cite[Proposition 5.1 and Theorem 5.2]{Rango1}).

\begin{proposition}
	Let $V$ be a closed subspace of $\Hol(\Dc)$. Then, $V$ is $\Aff$-invariant if and only if $V$ is the closure of $\bigoplus_{k\in\N} [\Pc_k(\C^n)\otimes \Pc^{h_k}(\C)]$ in $\Hol(D)$, where $h_0\in \N\cup \Set{\infty}$ and $h_{k+1}\in \Set{h_k,(h_k-1)_+}$ for every $k\in\N$.
\end{proposition}

\begin{theorem}
	Take $\lambda\in \R$ and let $H$ be a non-trivial strongly decent and saturated semi-Hilbert subspace of $\Hol(\Dc)$ such that   $U_\lambda$ induces a bounded  (resp.\ isometric) representation of $\Aff$ in $H$. Then, either one of the following conditions hold:
	\begin{itemize}
		\item there is  $k\in \N$ such that $s+2k>0$ and  such that $H$ is a dense subspace of $H_{\lambda,k}(\Dc)=\Set{f\in \Hol(\Dc)\colon \partial_2^k f \in H_{\lambda+2 k}(\Dc)}$, endowed with an equivalent (resp.\ proportional) seminorm;
		
		\item $\lambda\in -\N$ and there is a non-empty subset $J$ of $\N\cap (-s-2 \N)$ such that $H$ contains $\bigoplus_{j\in J} [\Pc_j(\C^n)\otimes \Pc_{-(s+j)/2}(\C)]$ as a dense  subspace, and its seminorm is equivalent (resp.\ proportional) thereon to the following one:
		\[
		\norm*{\sum_{j\in J}  (p_j\otimes (\,\cdot\,)^{-(s+j)/2})}^2=\sum_{j\in J} \norm{p_j}^2_j
		\]
		for every $\sum_{j\in J}  (p_j\otimes (\,\cdot\,)^{-(s+j)/2})\in \bigoplus_{j\in J} [\Pc_j(\C^n)\otimes \Pc_{-(s+j)/2}(\C)]$, where $\norm{\,\cdot\,}_j$ is some $U(\Phi)$-invariant Hilbert norm on $\Pc_j(\C^n)$.
	\end{itemize}
\end{theorem}

Also in this case one may describe the $G(\Dc)$-$U_\lambda$-invariant spaces   (cf.~\cite[Proposition 5.1 and Theorem 5.3]{Rango1} and also~\cite[Theorem 5.5]{Arcozzietal}).

\begin{proposition}
	Let $V$ be a closed subspace of $\Hol(\Dc)$ and take $\lambda\in \R$. Then, $V$ is $G(\Dc)$-$U_\lambda$-invariant if and only if either $V=\Set{0}$, $V=\Hol(\Dc)$, or $\lambda\in-\N$ and  $V=\Pc^{1-\lambda}(\C^{n+1})$.
\end{proposition}

\begin{proposition}
	If $\lambda\in -\N$, then $\widetilde H_\lambda(\Dc)$ is the subspace of $H_{\lambda,1-\lambda}(\Dc)$ consisting of the $f\in \Hol(\Dc)$ such that $\partial_1^k f(0,\,\cdot\,)\in  H_{\lambda+k,(1-\lambda-k)_+}(\C_+)$ for every $k\in\N$ (with a proportional seminorm).
	
	In addition, $\widetilde H_0$ is the subspace of $H_{0,1}(\Dc)$ consisting of the $f\in \Hol(\Dc)$ such that $\nabla f(\zeta,z)\to 0$ for $\Im z\to+\infty$ and $\abs{\zeta}\meg R$, for every $R>0$.
\end{proposition}

\subsection{The General Case}

So far, there does not seem to be a reasonable description of $\widetilde H_\lambda(\Dc)$ when $n\neq 0$ and $r>1$. Nonetheless, we have the following partial description of the closed $U_\lambda$-invariant vector subspaces of $\Hol(\Dc)$ (cf.~\cite[Proposition 4.7]{Tubi}).

\begin{proposition}\label{prop:11}
	Take $\lambda\in \R$ and keep the notation of Proposition~\ref{prop:5}. Then, a closed vector subspace $V$ of $\Hol(\Dc)$  is $G(\Dc)$-$U_\lambda$-invariant if and only if $V=\Set{0}$, $V=\Hol(\Dc)$, or there is $k\in \Set{1,\dots, r}$ such that $\frac 1 2 a (k-1)-\lambda\in \N$ and 
	\[
	V=\Set{f\in \Hol(\Dc)\colon \forall \phi\in G(\Dc)\:\: (U_\lambda(\phi)f)* I^{-\vect s}=0},
	\]
	for every $\vect s\in N_{\lambda,k}$.
\end{proposition}

\section{Minimal and Maximal Invariant Spaces}\label{sec:7}

We keep the notation of Sections~\ref{sec:2} to~\ref{sec:6}. In this section we shall briefly describe the minimal and maximal $\lambda$-invariant spaces.

\begin{definition}
	Take $\lambda\in \R$. For every $k=0,\dots, q(\lambda)$, define finite subsets $N'_{\lambda,k}$ of $ \N_\Omega^*$ so that the closed $G(\Dc)$-$U_\lambda$-invariant \emph{proper} subspaces of $\Hol(F+i  \Omega)$ are precisely the  $\bigcap_{\vect s\in N'_{\lambda,k}}\ker \Dc_{\vect s}$, $k=0,\dots, q(\lambda)$, ordered increasingly.\footnote{This is possible setting $N'_{\lambda,0}=\Set{\vect 0}$ and using Proposition~\ref{prop:5}.}
	Then, set  
	\[
	V_{\lambda,k}\coloneqq \Set{f\in \Hol(\Dc)\colon \forall \phi\in G_0(\Dc)\:\: \forall \vect s\in N'_{\lambda,k}\:\: (U_\lambda(\phi)f)*I^{-\vect s}=0},
	\] 
	for $k=0,\dots, q(\lambda)$. 
	
	Then,   define, for every $k=0,\dots,q(\lambda)$, 
	\[
	\Ac^\infty_{\lambda,k}(\Dc)=\Set{f\in \Hol(\Dc)\colon \sup_{\phi\in G_0(\Dc)} \sup_{\vect s\in N'_{\lambda,k}} \abs{((U_\lambda(\phi)f)*I^{-\vect s})(0, i  e) }<\infty },
	\]
	endowed with the corresponding seminorm. We shall also simply write $\Ac^\infty_{\lambda,k}$ instead of $\Ac^\infty_{\lambda,k}(\Dc)$, and we shall define $\Ac^\infty_{\lambda,k}(D)\coloneqq \Cc_\lambda^{-1} \Ac^\infty_{\lambda,k}(\Dc)$.
\end{definition}

\begin{theorem}\label{theorem:3}
	Take $\lambda\in \R$ and $k\in \Set{0,\dots, q(\lambda)}$. Then, the following hold:
	\begin{enumerate}
		\item[\textnormal{(1)}] the space $\Ac^\infty_{\lambda,k}$ is strictly $\lambda$-invariant; in addition, 
		\begin{equation}\label{eq:6}
			\norm{f}_{\Ac^{\infty}_{\lambda,k}}=\sup_{\phi\in G(\Dc)} \sup_{(\zeta,z)\in \Dc} \sup_{\vect s\in N'_{\lambda,k}}\Delta^{(\lambda/2)\vect 1_r+\vect s}(\Im z-\Phi(\zeta)) \abs{ ((U_\lambda(\phi)f)*I^{-\vect s} )(\zeta,z) } 
		\end{equation}
		for every $f\in \Ac^\infty_{\lambda,k}$;
		
		\item[\textnormal{(2)}] if $X$ is a strongly decent and saturated semi-Banach subspace of $\Hol(\Dc)$, the closure of $\Set{0}$ in $X$ is $V_{\lambda,k}$ for some $k\in \Set{0,\dots, q(\lambda)}$,  and $U_\lambda$ induces a bounded ray representation of $G_0(\Dc)$ in $X$, then $X\subseteq \Ac^\infty_{\lambda,h}$ continuously for every $h=k,\dots,q(\lambda)$;

		\item[\textnormal{(3)}] if $X$ is a  decent  semi-Banach subspace of $\Hol(\Dc)$  and $U_\lambda$ induces a bounded ray representation of $G_0(\Dc)$ in $X$, then $X\subseteq \Ac^\infty_{\lambda,q(\lambda)}$ continuously;
		
		\item[\textnormal{(4)}] $\Ac^\infty_{\lambda,k}=V_{\lambda,k}$ with the trivial seminorm  if and only if $\lambda<0$ and either $k<q(\lambda)$ or $m/r-1-\lambda\not \in \N$;
		
		\item[\textnormal{(5)}] $\Ac^\infty_{\lambda,0}=\Set{f\in \Hol(\Dc)\colon \sup_{(\zeta,z)\in D}\Delta^{\lambda/2}(\Im z-\Phi(\zeta))\abs{f(\zeta,z)}<\infty}$, with the corresponding norm;
		
		\item[\textnormal{(6)}] if $n=0$ and $m/r-1-\lambda\in\N$, then  $\Ac^{\infty}_{\lambda,q(\lambda)}=\Set{f\in \Hol(\Dc)\colon \square^{m/r-\lambda}f\in \Ac^\infty_{2 m/r-\lambda,0}}$, with the corresponding seminorm.
	\end{enumerate}
\end{theorem}

Notice that $\Ac^\infty_{0,1}$ is Timoney's Bloch space, defined in~\cite{Timoney1}, and erroneously identified with $\Ac^\infty_{0,q(\lambda)}$ in~\cite{Timoney2} (cf.~also~\cite{Agranovski}).

\begin{proof}
	(1) It is clear, by construction, that $\Ac^\infty_{\lambda,k}$ is $U_\lambda$-invariant with its seminorm. 
	Then, take $h\in \Omega$, and observe that there is $t\in T_-$ (cf.~Subsection~\ref{sec:6:2}) such that $h=t\cdot e$. If $\psi\colon (\zeta,z)\mapsto (t^{-1}\cdot \zeta, t^{-1}\cdot z)$, then $\abs{J \psi^{-1}(0,i  e)}=\Delta^{g/2}(h)$ (by the invariance of the Bergman kernel) and $t_* I^{-\vect s}=\Delta^{\vect s}(h) I_{\lambda,k}$ (by the definition of $I^{-\vect s}$ by means of the Laplace transform), so that
	\[
	\abs{((U_\lambda(\psi) f)*I^{-\vect s} )(0,i e)}= \Delta^{(\lambda/2)\vect 1_r+\vect s}(h) \abs{(f*I^{-\vect s} )(0,i h)}.
	\]
	In a similar way, if we fix $(\zeta,z)\in\Dc$, chhose $h=\Im z- \Phi(\zeta)$ in the above computations, and  choose $\psi'$ as the action of $(-\zeta,-\Re z)\in \Nc$ on $\Dc$, then
	\[
	\abs{((U_\lambda(\psi'\psi) f)*I^{-\vect s})(0,i e)}= \Delta^{(\lambda/2)\vect 1_r+\vect s}(\Im z-\Phi(\zeta)) \abs{(f*I^{-\vect s} )(\zeta,z)},
	\]
	whence~\eqref{eq:6} using the fact that $G(\Dc)=G_0(\Dc) \Aff(\Dc)$ (cf., e.g.,~\cite[Remark 1]{Nakajima}).
	
	Now, observe that the closure of $\Set{0}$ in $\Ac^\infty_{\lambda,k}$ is $V_{\lambda,k}$ by construction, so that   $\Ac^\infty_{\lambda,k}$ is strongly decent and saturated if and only if the canonical mapping $\Ac^\infty_{\lambda,k}\to \Hol(\Dc)/V_{\lambda,k}$ is continuous. Since the seminorm of $\Ac^\infty_{\lambda,k}$ is clearly lower semi-continuous on $\Hol(\Dc)$, it is clear that $\Ac^\infty_{\lambda,k}$ is strongly decent and saturated  if and only if $\Ac^\infty_{\lambda,k}$ is complete.
	Let us then prove that $\Ac^\infty_{\lambda,0}$ is complete. 
	First observe that there are $\phi_1,\dots, \phi_N$ in the stabilizer of $e$ in $G_0(\Omega)$ such that convolution (on the right) with the point distributions $(\phi_1)_* I^{-\vect s},\dots,(\phi_N)_* I^{-\vect s}$ forms a basis of  $\Dc_{\vect s}$ for every $\vect s\in N'_{\lambda,k}$. 
	Therefore, using~\eqref{eq:6}, we see that   the mappings $\Ac^\infty_{\lambda,k}\ni f \mapsto [f(0,\,\cdot\,)]*(\phi_j)_* I^{-\vect s}\in \Hol(T_\Omega) $ are  continuous for every $j=1,\dots, N$ and for every $\vect s\in N'_{\lambda,k}$, where $T_\Omega\coloneqq F+i \Omega$. 
	Then, applying~\cite[Theorem 7.6.13]{Hormander} to the differential operators $f \mapsto f*(\phi_j)_* I^{-\vect s}$, $j=1,\dots,N$, $\vect s\in N_{\lambda,k}'$, and  $\partial_v-i \partial_{i v} $, with $v$ in a basis of $F$, we see that the mapping $\Hol(T_\Omega)f\ni \mapsto (f*(\phi_j)_* I^{-\vect s})_{j,\vect s}\in \Hol(T_\Omega)^{\Set{1,\dots, N}\times N'_{\lambda,k}}$ has a closed image, so that it induces an isomorphism of $\Hol(T_\Omega)/V_{\lambda,k}(T_\Omega)$ onto its image in $\Hol(T_\Omega)^{\Set{1,\dots, N}\times N'_{\lambda,k}}$, where $V_{\lambda,k}(T_\Omega)=\bigcap_{\vect s\in N'_{\lambda,k}}\ker \Dc_{\vect s}=\Set{f\in \Hol(T_\Omega)\colon [(\zeta,z)\mapsto f(z)]\in V_{\lambda,k}}$ (cf.~\cite[Proposition 4.7]{Tubi}). 
	Therefore, the mapping $\Rc\colon \Ac^\infty_{\lambda,k}\ni f \mapsto f(0,\,\cdot\,)\in \Hol(T_\Omega)/V_{\lambda,k}(T_\Omega)$ is continuous. Now, observe that the results of Section~\ref{sec:2} imply that there are two unique $K(\Dc)$- and $K(T_\Omega)$-invariant continuous projectors of $\Hol(\Dc)$ and $\Hol(T_\Omega)$ with kernels $V_{\lambda,k}$ and $V_{\lambda,k}(T_\Omega)$, respectively, and that
	\[
	(P f)(0,\,\cdot\,)=P_0 [f(0,\,\cdot\,)]
	\]
	for every $f\in \Hol(\Dc)$, where $K(\Dc)$ and $K(T_\Omega)$ denote the stabilizers of $(0, i e)$ and $i e $ in $G(\Dc)$ and $G(T_\Omega)$, respectively.\footnote{With the notation of Section~\ref{sec:2}, $Pf=  \Cc_\lambda \sum_{q(\vect s,\lambda)\Meg k} \pi_{\vect s}( \Cc_\lambda^{-1}f) $ for every $f\in \Hol(\Dc)$.} 
	Therefore, the mapping
	\[
	P_0 \Rc=\Rc P \colon \Ac^\infty_{\lambda,k}\to \Hol(T_\Omega)
	\]
	is continuous. Hence, the mappings $\Rc P U_\lambda(\phi)=\Rc U_\lambda(\phi) P\colon \Ac^\infty_{\lambda,k}\to \Hol(T_\Omega)$, as $\phi$ runs through $K(\Dc)$, are equicontinuous. Since the $K(\Dc)$-orbit of $ T_\Omega$ is the whole of $\Dc$ (because of the so-called polar decomposition, cf.~\cite[Theorem 1.1 of Chapter IX]{Helgason}), every compact subset of $\Hol(\Dc)$ is contained in the $K(\Dc)$-orbit of some compact subset $L$ of $T_\Omega$ (namely, of the intersection of $T_\Omega$ with the $K(\Dc)$-orbit of $L$), so that the mapping $P\colon \Ac^\infty_{\lambda,k}\to \Hol(\Dc)$ is continuous by the previous remarks. If $(f_j)$ is a Cauchy sequence in $\Ac^\infty_{\lambda,k}$, then $(P f_j)$ is a Cauchy sequence in both $\Ac^\infty_{\lambda,k}$ and $\Hol(\Dc)$. Since the seminorm of $\Ac^\infty_{\lambda,k}$ is lower semicontinuous for the topology of $\Hol(\Dc)$, this implies that the limit of $(P f_j)$  in $\Hol(\Dc)$ belongs to $\Ac^\infty_{\lambda,k}$ and is a limit of $(P f_j)$ (hence of $(f_j)$, since $f_j-P f_j\in V_{\lambda,k}$ for every $j\in\N$) in $\Ac^\infty_{\lambda,k}$. Then, $\Ac^\infty_{\lambda,k}$ is complete.

	(2) Observe that the canonical mappings $X\to \Hol(\Dc)/V_{\lambda,k}\to \Hol(\Dc)/V_{\lambda,h}$ are continuous, so that the mapping $f \mapsto (f*I^{-\vect s})(0,i e)$ is continuous on $X$ for every $\vect s\in N'_{\lambda,h}$. Thus, there are two constants $C,C'>0$ such that
	\[
	\abs{((U_\lambda(\phi)f)*I^{-\vect s})(0,i e)}\meg C \norm{U_\lambda(\phi)f}_X\meg C' \norm{f}_X
	\]
	for every $f\in X$, for every $\vect s\in N'_{\lambda,k}$, and for every $\phi\in G_0(\Dc)$, whence our assertion.
	
	(3) The proof is similar to that of (2), since the decency of $X$ implies that $X\subseteq \Hol(\Dc)/V_{\lambda,q(\lambda)}$ continuously.
	
	(4) -- (5) Observe that (1) and the invariance properties of the unweighted Bergman kernel show that 
	\[
	\norm{f}_{\Ac^\infty_{\lambda,0}}=\sup_{(\zeta,z)\in \Dc} \Delta^{(\lambda/2)\vect 1_r}(\Im z-\Phi(\zeta))\abs{f(\zeta,z)}
	\]
	for every $(\zeta,z)\in \Dc$ and for every $\lambda\in \R$. This proves (5), and also shows that $\Ac^\infty_{\lambda,0}=\Set{0}$ if and only if $\lambda<0$, thanks to~\cite[Proposition 3.5]{CalziPeloso}. Thus, assume that $k>0$, and observe that~\eqref{eq:6} shows  that $\Ac^{\infty}_{\lambda,k}*I_{-\vect s}$ is contained in the Bergman space $\Set{f\in \Hol(\Dc)\colon \sup_{(\zeta,z)\in \Dc}\Delta^{(\lambda/2)\vect 1_r+\vect s}(\Im z-\Phi(\zeta))\abs{f(\zeta,z)}<\infty }$ for every $\vect s\in N'_{\lambda,k}$. If we choose $\vect s$ as the minimum of $ N'_{\lambda,k}$ (so that $s_1=0$ unless $k=q(\lambda)$ and $m/r-1-\lambda\in \N$), then this latter Bergman space is reduced to $\Set{0}$ if (and only if) $(\lambda/2)\vect 1_r+\vect s\not \Meg \vect 0 $, that is, $\lambda<0$ and either $k<q(\lambda)$ or $m/r-1-\lambda	\not \in\N$, thanks to~\cite[Proposition 3.5]{CalziPeloso} again. 
	We are thus reduced to proving that $\Ac^\infty_{\lambda,q(\lambda)}\neq V_{\lambda,q(\lambda)}$ for every $\lambda\in m/r-1-\N$, and that $\Ac^\infty_{\lambda,k}\neq V_{\lambda,k}$ for every $\lambda\Meg0$ and for every $k=0,\dots, q(\lambda)$. The former fact   follows from (3), which implies that $\widetilde H_\lambda(\Dc)\subseteq \Ac^\infty_{\lambda,q(\lambda)}$ continuously. 
	The latter fact may be proved observing that, by (1) and (3),  $\Ac^\infty_{\lambda,0}\subseteq \Ac^\infty_{\lambda,k}$ continuously, and that $\Ac^\infty_{\lambda,0}$ is dense in $\Hol(\Dc)$,\footnote{Indeed, $\Ac^\infty_{\lambda,0}\cap H_g(\Dc)$ is dense in $H_g(\Dc)$ by~\cite[Proposition 3.9]{CalziPeloso}, and $H_g(\Dc)$ is dense in $\Hol(\Dc)$.} hence is not contained in $V_{\lambda,k}$.
	
	(6) This follows from (5) and Proposition~\ref{prop:20}.
\end{proof}

\begin{definition}\label{def:3}
	Take $\lambda\in \R$ and $k\in \Set{0,\dots,q(\lambda)}$ so that $\Ac^\infty_{\lambda,k}\neq V_{\lambda,k}$. Then, we define $\widehat \Ac^1_{\lambda,k}(\Dc)$ as the image of the mapping
	\[
	L^1(\widetilde G(\Dc))\ni f \mapsto \int_{ \widetilde G(\Dc) } f(\phi)\widetilde U_\lambda(\phi) h\,\dd \phi\in \Hol(\Dc)/V_{\lambda,k-1}
	\]
	endowed with the induced topology, where $\widetilde G(\Dc)$ is endowed with a left Haar measure and $h$ is any element of $\Cc_\lambda(\Pc)\cap V_{\lambda,k+1}\setminus V_{\lambda,k}$ (setting $V_{\lambda,q(\lambda)+1}\coloneqq \Hol(\Dc)$).  We then define $\Ac^1_{\lambda,k}(\Dc)\coloneqq \Set{f\in \Hol(\Dc)\colon f+V_{\lambda,k-1}\in \widehat \Ac^1_{\lambda,k}(\Dc)}$. We shall also write $\Ac^1_{\lambda,k}$ instead of $ \Ac^1_{\lambda,k}(\Dc)$, and we shall define $\Ac^1_{\lambda,k}(D)\coloneqq \Cc_\lambda^{-1}\Ac^1_{\lambda,k}(\Dc)$.
	
	In addition, we define $\widetilde A^1_\lambda(\Dc)$, for $\lambda>2(m/r-1)$, as the image of the mapping
	\[
	\ell^1(J)\ni a \mapsto \sum_{j\in J} a_j B^{-\lambda}_{(\zeta_j,z_j)} \Delta^{\lambda/2}(\Im z_j-\Phi(\zeta_j))\in \Hol(D),
	\]
	endowed with the corresponding topology, where 
	\[
	B^{-\lambda}_{(\zeta',z')}(\zeta'',z'')\coloneqq \Delta^{-\lambda}\left(\frac{z-\overline{z'}}{2 i } -\Phi(\zeta,\zeta')\right)
	\]
	for every $(\zeta,z),(\zeta',z')\in \Dc$, and where $(\zeta_j,z_j)_{j\in J}$ is a family of elements of $\Dc$ which is maximal for the relation $d((\zeta_j,z_j),(\zeta_{j'},z_{j'}))\Meg 2 \delta$, where $\delta$ is sufficiently small and $d$ is the Bergman metric (or, more generally, any Riemannian distance associated with a $G(\Dc)$-invariant complete Riemannian metric).
\end{definition}

Observe that $B^{-\lambda}$ is, up to a constant, the power to the exponent $\lambda/g$ of the unweighted Bergman kernel, so that it is $\widetilde U_\lambda(\phi)\otimes \overline{\widetilde U_\lambda(\phi)}$-invariant for every $\phi\in \widetilde G(\Dc)$.

As shown in~\cite[Chapter 5]{CalziPeloso}, combined with~\cite[Theorem 4.5]{Paralipomena}, the space $\widetilde A^1_\lambda(\Dc)$ and its topology do \emph{not} depend on the choice of the family $(\zeta_j,z_j)$, provided that $\delta$ is sufficiently small. In fact, the associated norms are uniformly equivalent as long as $\delta$ is fixed, as one sees combining the proof of~\cite[Theorem 4.5]{Paralipomena} with~\cite[Theorem 3.22]{CalziPeloso}. Actually, the space $\widetilde A^1_\lambda(\Dc)$ may be defined as the space of holomorphic extensions of a suitable Besov space defined on the \v Silov boundary of $\Dc$, and coincides with the usual Bergman space
\[
\Set{f\in \Hol(\Dc)\colon \int_{\Dc} \abs{f(\zeta,z)} \Delta^{\lambda/2}(\Im z-\Phi(\zeta))\,\dd \nu_\Dc(\zeta,z)<\infty }
\]
when $\lambda>2(g-1)$.

\begin{theorem}\label{theorem:2}
	Take $\lambda\in \R$ and $k\in \Set{0,\dots,q(\lambda)}$ so that $\Ac^\infty_{\lambda,k}\neq V_{\lambda,k}$.  Then, the following hold:
	\begin{enumerate}
		\item[\textnormal{(1)}] the space $\Ac^1_{\lambda,k}$ does not depend on the choice of $h$ and is $\lambda$-invariant, separable, and ultradecent;

		\item[\textnormal{(2)}] if $X$ is a strongly decent, saturated, separable semi-Banach subspace of $\Hol(\Dc)$ which is not contained in $V_{\lambda,k}$ and in which $U_\lambda$ induces a \emph{continuous} bounded representation of $G_0(\Dc)$, then $\Ac^1_{\lambda,k}\subseteq X$ continuously;\footnote{The assumption that $U_\lambda$ induces a \emph{continuous} representation of $G_0(\Dc)$ in $X$ may be replaced with the assumption that $X$ be strongly decent and that the closed unit ball in $X$ be closed in $\Hol(\Dc)$, with a slightly different argument. With these assumptions, the space $\Ac^\infty_{\lambda,k}$ is an eligible choice for $X$.}
		
		\item[\textnormal{(3)}] if $\lambda>2(m/r-1)$, then $\Ac^1_{\lambda,0}=\widetilde A^1_{\lambda}(\Dc)$;
		
		\item[\textnormal{(4)}] if $n=0$ and $m/r-1-\lambda\in\N$, then $\Ac^1_{\lambda,q(\lambda)}=\Set{f\in \Hol(\Dc)\colon \square^{m/r-\lambda} f\in \Ac^1_{2 m/r-\lambda,0} }$.
	\end{enumerate}
\end{theorem}

\begin{proof}
	(1) -- (2) We may assume that $X\neq \Hol(\Dc)$. Then, choose $k\meg k'\meg q(\lambda)$ so that the closure of $\Set{0}$ in $X$ is $V_{\lambda,k'}$.  Observe that  Proposition~\ref{prop:7} implies that $\widetilde U_\lambda$ induces a continuous representation of $\widetilde G(\Dc)$ in $X$, and that $X$ contains all the $h'\in \Cc_\lambda^{-1}(\Pc)$ which belong to the closure of $X$ in $\Hol(\Dc)$ (in particular, to $V_{\lambda,k'+1} $). 
	In particular, $h\in X$ and the mapping $\widetilde G(\Dc)\ni\phi \mapsto \widetilde U_\lambda(\phi)h\in X$ is continuous and bounded. Since $X$ embeds continuously in $\Hol(\Dc)/V_{\lambda,k'}$, this shows that the linear mapping
	\[
	L^1(\widetilde G(\Dc))\ni f \mapsto \int_{\widetilde G(\Dc)}  f(\phi)\widetilde U_\lambda(\phi) h\,\dd \phi\in X\subseteq\Hol(\Dc)/V_{\lambda,k'-1}
	\]
	is well defined and continuous, where the integral may be considered as an integral with values in $X$ or in $\Hol(\Dc)/V_{\lambda,k'-1}$. Thus, once we show that $\Ac^1_{\lambda,k}$ is well defined, this will complete the proof of (2). 
	
	If we take $X=\Ac^\infty_{\lambda,k}$, then $k=k'$,\footnote{Notice that $\Ac^\infty_{\lambda,k}$ is \emph{not} separable, but still contains $\Cc_\lambda^{-1}(\Pc)\cap \overline{\Ac^\infty_{\lambda,k}}$; this is true essentially because the closed unit ball in $\Ac^\infty_{\lambda,k}$ is also closed in $\Hol(\Dc)/V_{\lambda,k}$, so that the integrals of the functions with values in   $\Hol(\Dc)/V_{\lambda,k}$, whose image is  contained and bounded in $\Ac^\infty_{\lambda,k}$, actually belong to $\Ac^\infty_{\lambda,k}$.} so that the above computations show that the mapping defining $\Ac^1_{\lambda,k}$ is continuous with values in ($\Ac^\infty_{\lambda,k}$, hence in) $\Hol(\Dc)/V_{\lambda,k}$, so that $\Ac^1_{\lambda,k}$ is well defined. Further, since $\widehat\Ac^1_{\lambda,k}(\Dc)$ embeds continuously in $\Hol(\Dc)/V_{\lambda,k}$, and since $h$ induces a non-zero element of $\widehat\Ac^1_{\lambda,k}(\Dc)$, it is also clear that the closure of $\Set{0}$ in $\Ac^1_{\lambda,k}$ is precisely $V_{\lambda,k-1}$, so that $\Ac^1_{\lambda,k}$ is strongly decent and saturated. Since $L^1(\widetilde G(\Dc))$ is separable and complete, the same holds for $\Ac^1_{\lambda,k}$. Analogously, since the action of $\widetilde U_\lambda$ in $\Ac^1_{\lambda,k}$ is precisely the quotient of the left regular representation in $L^1(\widetilde G(\Dc))$, it is clear that $\widetilde U_\lambda$ induces a continuous and isometric representation of $\widetilde G(\Dc)$ in $\Ac^1_{\lambda,k}$, so that this latter space in ultradecent by Proposition~\ref{prop:7}. To see that $\Ac^1_{\lambda,k}$ is actually $G(\Dc)$-$U_\lambda$-invariant with its seminorm, it suffices to observe that $G(\Dc)=G_0(\Dc) K(\Dc)$, where $K(\Dc)=\Set{\phi\in G(\Dc)\colon \phi(0, i  e)=(0,i  e)}$ (cf.~the proof of~\cite[Proposition 4.6]{Tubi}),  that the natural action of the compact group $K(\Dc)$ on $\widetilde G(\Dc)$, induced by the action $k\cdot \phi\mapsto k \phi k^{-1}$ of $K(\Dc)$ on $G_0(\Dc)$, necessarily preserves the left Haar measure of $\widetilde G(\Dc)$, and that $U_\lambda(k) \widetilde U_\lambda(\phi)=\alpha_{k,\phi}\widetilde U_\lambda(k\cdot \phi) U_\lambda(k)$, with $\alpha_{k,\phi}\in \T$, for every $k\in K(\Dc)$ and for every $\phi\in \widetilde G(\Dc)$, once a representative of $U_\lambda(k)$ has been chosen.
	
	Finally, $\Ac^1_{\lambda,k}$ does not depend on the choice of $h$: if $A'$ is the space defined as $\Ac^1_{\lambda,k}$ starting with a different $h'$, then (2) (and (1) applied to $A'$) shows that $\Ac^1_{\lambda,k}\subseteq A'$ continuously, whereas (1) (and (2) applied to $A'$) shows that $A'\subseteq \Ac^1_{\lambda,k}$ continuously.
	
	(3) Observe first that $B^{-\lambda}_{(0,i  e)}=c\Cc_\lambda^{-1}(1)$ for some $c\neq 0$, by the invariance properties of the unweighted Bergman kernel, and $\widetilde U_\lambda$ induces a bounded representation of $\widetilde G(\Dc)$ in $\widetilde A^1_\lambda(\Dc)$, since the invariance properties of the unweighted Bergman kernel show that
	\[
	\widetilde U_\lambda(\phi) B^{-\lambda}_{(\zeta,z)} \Delta^{\lambda/2}(\Im z-\Phi(\zeta))=  \alpha B^{-\lambda}_{\phi(\zeta,z)} \Delta^{\lambda/2}(\Im \phi(\zeta,z)_2- \Phi(\phi(\zeta,z)_1))
	\]
	for every $(\zeta,z)\in \Dc$ and for every $\phi\in \widetilde G(\Dc)$, where $\alpha=\frac{\overline{(J \phi )(\zeta,z)^\lambda}}{\abs{(J \phi)(\zeta,z)}^\lambda}$ has modulus $1$ (and since we remarked earlier that the norms defined by the various $(\phi(\zeta_j,z_j))$, given $(\zeta_j,z_j)$ as in Definition~\ref{def:3}, are uniformly equivalent). In addition, the space generated by the $B^{-\lambda}_{(\zeta,z)}$, $(\zeta,z)\in \Dc$,  is contained and dense in $\widetilde A^1_\lambda(\Dc)$ by the previous remarks. One may then show (cf.~the proof of~\cite[Lemma 5.15]{CalziPeloso} and the previous remarks) that $\widetilde U_\lambda$ acts continuously on the  $B^{-\lambda}_{(\zeta,z)}$, hence induces a continuous representation in $\widetilde A^1_\lambda(\Dc)$. Thus, (2) shows that $\Ac^1_{\lambda,0}\subseteq \widetilde A^1_\lambda(\Dc)$ continuously. The converse inclusion is trivial, since $\Ac^1_{\lambda,0}$ contains the image of 
	\[
	\cM^1(\widetilde G(\Dc))\ni f \mapsto \int_{\widetilde G(\Dc)}  \widetilde U_\lambda(\phi) B^{-\lambda}_{(0, i  e)}\,\dd \mi(\phi)\in\Hol(\Dc),
	\]
	whence our claim considering the continuous linear mapping
	\[
	\ell^1(J)\ni (a_j)\mapsto \sum_{j\in J} a_j \delta_{\phi_j}\in \cM^1(\widetilde G(\Dc)),
	\]
	where $\phi_j\colon (\zeta,z)\mapsto (\zeta_j, \Re z_j+ i  \Phi(\zeta_j))\cdot (t_j\cdot \zeta , t_j\cdot z)$, $t_j\in T_-$ is such that $t_j\cdot e=\Im z_j-\Phi(\zeta_j)$, and $(\zeta_j,z_j)$ is chosen as in Definition~\ref{def:3}.
	
	(4) This follows from (3) and Proposition~\ref{prop:20}.
\end{proof}

\begin{proposition}
	Take $\lambda>m/r-1$. Then, the sesquilinear mapping
	\[
	\Ac^1_{\lambda,0}(D)\times \Ac^\infty_{\lambda,0}(D)\ni (f,h)\mapsto \lim_{R\to 1^-} \langle f(R\,\cdot\,)\vert h(R\,\cdot\,)\rangle_{H_\lambda(D)}
	\]
	induces an antilinear isometric isomorphism of $\Ac^\infty_{\lambda,0}(D)$ onto the dual of $\Ac^1_{\lambda,0}(D)$.
\end{proposition}

When $n=0$ and $m=1$, a different version of this result is essentially contained in~\cite[Theorem 3]{AlemanMas}. This result was also stated (allowing $\lambda\in \Wc(\Omega)$) in~\cite[Proposition 3.5]{ArazyUpmeier3}. Nonetheless, the statement is incorrect (as one realizes taking the extreme case $\lambda=0$, in which case the statement claims that the dual of the space of constant functions may be identified with the space of bounded holomorphic functions) and the proof is incomplete (as no reference is made to the problem of extending the scalar product of $H_\lambda(D)$ to $\Ac^1_{\lambda,0}(D)\times\Ac^\infty_{\lambda,0}(D)$ -- which is non-trivial, since $H_\lambda(D)$ is \emph{not} dense in $\Ac^\infty_{\lambda,0}(D)$).

If $\lambda>2(m/r-1)$, then the identifications of $\Ac^1_{\lambda,0}$ and $\Ac^\infty_{\lambda,0}$ provided in (3) of Theorem~\ref{theorem:2} and in (5) of Theorem~\ref{theorem:3}, combined with~\cite[Proposition 5.12]{CalziPeloso},\footnote{In this latter result, one should replace $\vect s''$ with $\vect s''-\vect b-\vect d$ in order to get the correct statement.} provide a different interpretation of the duality between $\Ac^1_{\lambda,0}(\Dc)$ and $\Ac^\infty_{\lambda,0}(\Dc)$.

\begin{proof}
	First observe that, by means of (5) of Theorem~\ref{theorem:3} we see that 
	\[\Ac^\infty_{\lambda,0}(D)=\Set{f\in \Hol(D)\colon \sup_{z\in D} \abs{f(z)} \Kc^{-\lambda/(2 g)}(z,z)<\infty}.
	\]
	Then, define $\Ff\colon \Ac^1_{\lambda,0}(D)\to \Hol(D)$ as in Proposition~\ref{prop:12}, and observe that $\Ff(L)(\phi(0))=\overline{(J \phi)^{-\lambda/g}(0)} \langle L, \widetilde U_\lambda(\phi) 1\rangle $  and that $\abs{(J \phi)(0)}^{-\lambda/g}=\Kc^{\lambda/(2g)}(\phi(0),\phi(0))$ for every $\phi\in \widetilde G(D)$, by the invariance properties of the unweighted Bergman Kernel $\Kc$. Therefore, $\norm{\Ff(L)}_{\Ac^\infty_{\lambda,0}(D)}\meg \norm{L}_{\Ac^1_{\lambda,0}(D)'}$. Further, if $f\in L^1(\widetilde G(D))$, then
	\[
	\begin{split}
		&\Big\langle L, \int_{\widetilde G(D)} f(\phi) \widetilde U_\lambda(\phi) 1\,\dd \phi\Big\rangle= \int_{\widetilde G(D)} f(\phi)\langle L, \widetilde U_\lambda(\phi) 1\rangle\,\dd \phi\\
		&\qquad=\int_{\widetilde G(D)} f(\phi) \frac{\overline{(J\phi)^{\lambda/g}(0)}}{\abs{(J\phi)(0)}^{\lambda/g}} \Kc^{-\lambda/(2g)}(\phi(0),\phi(0)) \Ff(L)(\phi(0)) \,\dd \phi,
	\end{split}
	\]
	so that $\norm{\Ff(L)}_{\Ac^\infty_{\lambda,0}(D)}= \norm{L}_{\Ac^1_{\lambda,0}(D)'}$. Finally, it is readily seen that every $h\in \Ac^\infty_\lambda(D)$ induces a continuous linear functional on $\Ac^1_\lambda(D)$, so that the assertion follows.
\end{proof}

Applying Proposition~\ref{prop:20}, we also get the following corollary.

\begin{corollary}
	Take $\lambda\in m/r-1-\N$, and assume that $n=0$. Then, the continuous sesquilinear mapping
	\begin{eqnarray*}
		&\Ac^1_{\lambda,q(\lambda)}(D)\times \Ac^\infty_{\lambda,q(\lambda)}(D)\ni & (f,h)\mapsto \lim_{R\to 1^-} \langle f\vert g\rangle_{\widetilde H_\lambda}\\ & &= \lim_{R\to 1^-} \langle \square^{m/r-\lambda }f(R\,\cdot\,)\vert \square^{m/r-\lambda }h(R\,\cdot\,)\rangle_{H_{2m/r-\lambda}(D)}, 
	\end{eqnarray*}
	induces  an isometric antilinear isomorphism of $\Ac^\infty_{\lambda,q(\lambda)}/V_{\lambda,q(\lambda)}$ onto the dual of $\Ac^1_{\lambda,q(\lambda)}$.
\end{corollary}

When $n=0$ and $m=1$, this result was proved in~\cite[pp.~122--124]{ArazyFisherPeetre}. This result was later extended to the case $n>0$ and $m=1$ in~\cite[Theorem 5.13]{Peloso}.

\end{document}